\documentclass[12pt,reqno]{amsart}

\usepackage[T1]{fontenc}
\pdfoutput=1
\usepackage{mathptmx,microtype,graphicx,dsfont,booktabs,mathtools}
\usepackage[hidelinks]{hyperref}
\usepackage{amsthm,amsmath,amssymb}
\usepackage[foot]{amsaddr}
\usepackage{enumitem}
\usepackage[font=footnotesize,margin=2cm]{caption}
\usepackage{cite}
\theoremstyle{plain}
\newtheorem{theorem}{Theorem}
\newtheorem{lemma}[theorem]{Lemma}

\newtheorem{definition}[theorem]{Definition}
\newtheorem*{remark}{Remark}

\newcommand{\cD}{\mathcal{D}}

\renewcommand{\P}{{\mathbb P}}
\newcommand{\E}{{\mathbb E}}

\newcommand{\R}{{\mathbb R}}
\newcommand{\Z}{{\mathbb Z}}
\newcommand{\cS}{{\mathcal S}}
\newcommand{\cG}{{\mathcal G}}

\newcommand{\cX}{{\mathcal X}}
\newcommand{\cT}{{\mathcal T}}
\newcommand{\cE}{{\mathcal E}}
\newcommand{\cK}{{\mathcal K}}

\newcommand{\cI}{{\mathcal I}}

\newcommand{\cP}{{\mathcal P}}

\renewcommand{\cD}{{\mathcal D}}
\newcommand{\bbZ}{{\mathbf Z}}

\newcommand{\Nex}{\hat N}

\newcommand{\tour}{{\rm Tour}}
\newcommand{\score}{{\rm Score}}
\newcommand{\win}{{\rm Win}}
\newcommand{\ig}{{\rm IntGr}}

\renewcommand{\le}{\leqslant}
\renewcommand{\ge}{\geqslant}

\author[M. Buckland]{Matthew Buckland}
\address{Department of Statistics, University of Oxford}
\email{matthew.buckland@bnc.ox.ac.uk}

\author[B. Kolesnik]{Brett Kolesnik}
\address{Department of Statistics, University of Warwick}
\email{brett.kolesnik@warwick.ac.uk}

\author[R. Mitchell]{Rivka Mitchell}
\address{Department of Mathematics, University of Oxford}
\email{rivka.mitchell@maths.ox.ac.uk}

\author[T. Przyby{\l}owski]{Tomasz Przyby{\l}owski}
\address{Department of Mathematics, University of Oxford}
\email{przybylowski@maths.ox.ac.uk}

	%Paired and multiple comparisons

\begin{document}

\title
[Random walks on Coxeter interchange graphs]
{Random walks on Coxeter interchange graphs}

\begin{abstract}
A tournament is an orientation of a graph. 
Vertices are players and edges are games, 
directed away from the winner. 
Kannan, Tetali and Vempala and McShine 
showed that tournaments with 
given score sequence can be rapidly sampled, 
via simple random walks on 
the interchange graphs of Brualdi and Li. 
These graphs are generated by 
the cyclically directed triangle, 
in the sense that traversing an edge corresponds to 
the reversal of such a triangle in a tournament.

We study Coxeter tournaments on 
Zaslavsky's signed graphs. 
These tournaments involve collaborative and solitaire games, 
as well as the usual competitive games. 
The interchange graphs are 
richer in complexity, as a variety of other 
generators are involved. 
We prove rapid mixing by an intricate application of 
Bubley and Dyer's method of path coupling, using 
a delicate re-weighting of the graph metric.
Geometric connections 
with the Coxeter permutahedra 
introduced by Ardila, Castillo, Eur and Postnikov
are discussed.  
\end{abstract}

\maketitle

%%%%%%%%%%%%%%%%%%%%%%%%%%%%%%%%%%%%%%
%%%%%%%%%%%%%%%%%%%%%%%%%%%%%%%%%%%%%%
%%%%%%%%%%%%%%%%%%%%%%%%%%%%%%%%%%%%%%
%%%%%%%%%%%%%%%%%%%%%%%%%%%%%%%%%%%%%%
%%%%%%%%%%%%%%%%%%%%%%%%%%%%%%%%%%%%%%
%%%%%%%%%%%%%%%%%%%%%%%%%%%%%%%%%%%%%%
\section{Introduction}\label{S_intro}

A tournament is an orientation of a graph. We think of 
vertices as players and edges as games, the orientation of which
indicates the winner. 
Tournaments are related to the geometry of
the permutahedron $\Pi_{n-1}$, which is a classical polytope
in discrete convex geometry. 
See, e.g., 
Stanley \cite{Sta80}, Ziegler \cite{Zie95}, 
and Kolesnik and Sanchez \cite{KS20}.

Classical combinatorics is related to the root system of 
type $A_n$. Coxeter combinatorics is concerned 
with extensions to the other roots systems of types 
$B_n$, $C_n$ and $D_n$ (and sometimes also the
finite, exceptional types 
$E_6$, $E_7$, $E_8$, $F_4$ and $G_2$). 
For example, works by 
Galashin, Hopkins, McConville and
Postnikov \cite{GHMcSP19,GHMcSP21}
have investigated Coxeter versions of 
the chip-firing game (the sandpile model). 

Recently, Kolesnik and Sanchez \cite{KS23} 
introduced the Coxeter 
analogue of graph tournaments, which are 
associated with orientations
of signed graphs, as in Zaslavsky \cite{Zas91}, 
and the Coxeter permutahedra
$\Pi_\Phi$, recently introduced by 
Ardila, Castillo, Eur and Postnikov \cite{ACEP20}. 
Coxeter tournaments involve
collaborative and solitaire games, 
as well as the usual competitive games in graph tournaments. 

In this work, we show 
(see Theorem \ref{T_Main} below)
that random walks  
rapidly mix on 
the sets of Coxeter tournaments 
with given score sequence, that is, 
on 
the fibers of the Coxeter permutahedra
$\Pi_\Phi$. Informally, this means that the walk is close
to uniform in the fiber after a short amount of time, 
yielding an efficient way to sample from this set of interest. 

Many combinatorial properties of these structures remain mysterious. 
The purpose of this work is to explore the associated 
Coxeter interchange graphs, which encode their combinatorics, 
via random walks. 
These graphs, introduced by Kolesnik, 
Mitchell and Przyby{\l}owski  \cite{KMP23}, 
generalize the interchange graphs introduced by 
Brualdi and Li \cite{BL84}. 
Rapid mixing in the classical setting was established by 
Kannan, Tetali and Vempala \cite{KTV99} and McShine \cite{McS00}.
We recover these results by our general strategy. 

Let us emphasize that even the classical 
interchange graphs appear to be difficult to describe in general.  
Indeed,  
Brualdi and Li \cite[p.\ 151]{BL84}
state that they have 
``a rich and fascinating combinatorial structure
and that much remains to be determined.'' 
Even counting the number
of vertices is of
``considerable interest and considerable difficulty'' \cite[p.\ 143]{BL84}.

Beginning with Spencer \cite{Spe74}, and subsequent works by 
McKay \cite{McKay90}, McKay and Wang \cite{McKW96}, 
and Isaev, Iyer and McKay \cite{IIMcK20}, 
asymptotic estimates for the number of vertices in the interchange
graphs have been found only 
for fibers of points near the center of $\Pi_{n-1}$. 
In the other extreme, Chen, Chang and Wang 
 \cite{CCW09} showed that, for certain points near the boundary
 of $\Pi_{n-1}$, 
the interchange graph is the classical hypercube. 

The Coxeter 
interchange graphs are richer still. 
Therefore, in broad terms, we show in this work
that random walks rapidly mix
on a wide and intricate class of graphs. 

Rapid mixing can sometimes be used 
to approximately count sets of interest. 
Roughly speaking, this is because the uniform measure $\pi$
on a set $S$ is related to the size of the set
$\pi=1/|S|$. 
See, e.g., Sinclair \cite{Sin93} for more details. 
The current work might serve as a first step 
towards developing efficient approximate counting 
schemes for the fibers of the Coxeter permutahedra 
$\Pi_\Phi$.

As this work touches on a variety of subjects (combinatorics, 
geometry, algebra and probability), some    
preliminaries are required before we can state our results precisely. 
In Section \ref{S_context}, 
we discuss the literature related to 
tournaments and 
the standard
permutahedron $\Pi_{n-1}$ (of type $A_{n-1}$). 
Our results are discussed informally in 
Sections \ref{S_purpose} and \ref{S_discussion}. 
Further background on tournaments, root systems, signed graphs
and combinatorial geometry is in 
Section \ref{S_back} and in the previous works 
in this series \cite{KS20,KS23,KMP23}.
Our main result is stated formally in Section \ref{S_results}.
See Sections \ref{S_DecompRev}, 
\ref{S_Networks} and \ref{S_Coupling}
for the proofs. 
Finally, a number of open problems and future
directions are listed in Section \ref{S_OPs}.

We hope that this work will 
serve as an invitation 
to step into the 
Coxeter ``worlds'' (of types $B_n$, $C_n$ and $D_n$). 
We believe that 
Coxeter combinatorics 
is fertile ground, where
algebraists, combinatorialists, 
geometers and probabilists
can open new lines of fruitful communication. 
In particular, 
many problems in discrete probability
likely have a Coxeter analogue, 
waiting to be discovered.

%%%%%%%%%%%%%%%%%%%%%%%%%%%%%%%%%%%%%%
%%%%%%%%%%%%%%%%%%%%%%%%%%%%%%%%%%%%%%
%%%%%%%%%%%%%%%%%%%%%%%%%%%%%%%%%%%%%%
%%%%%%%%%%%%%%%%%%%%%%%%%%%%%%%%%%%%%%
%%%%%%%%%%%%%%%%%%%%%%%%%%%%%%%%%%%%%%
%%%%%%%%%%%%%%%%%%%%%%%%%%%%%%%%%%%%%%
\subsection{Context}
\label{S_context}

A {\it tournament} is an orientation of the 
complete graph $K_n$, encoded as some 
$T=(w_{ij}:i>j)$ with all $w_{ij}\in\{0,1\}$. 
Each edge $\{i,j\}$ in $K_n$ is oriented as 
$i\to j$ if $w_{ij}=1$ or $i\leftarrow j$
if $w_{ij}=0$. 
We think of each edge as a game, 
directed away from the winner. 
The {\it win sequence}
\[
{\bf w}(T)=\sum_{i>j}[w_{ij}{\bf e}_i+(1-w_{ij}){\bf e}_j]
\]
lists the total number of wins by each player, 
where ${\bf e}_i\in\Z^n$ are the standard basis vectors. 
We let 
${\bf w}_n=(0,1,\ldots,n-1)$ denote the {\it standard win sequence,}
corresponding to the {\it transitive} (acyclic) tournament 
in which $w_{ij}=1$ for all $i>j$. 
In a sense, ${\bf w}_n$ and its permutations
are as ``spread out'' as possible.  

\begin{figure}[h!]
\centering
\includegraphics[scale=0.9]{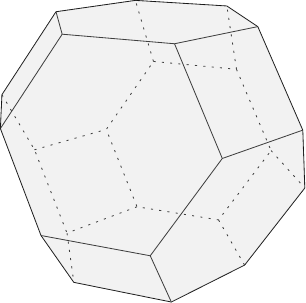}
\caption{The permutahedron 
$\Pi_3\subset\R^4$, projected into $\R^3$. 
Its 24 vertices correspond to the permutations of 
the standard win sequence ${\bf w}_4=(0,1,2,3)$.
}
\label{F_perm}
\end{figure}

Results by Rado \cite{Rad52} and Landau \cite{Lan53}
imply that the set $\win(n)$ of all win sequences is precisely the 
set of lattice points in the permutahedron $\Pi_{n-1}$, 
that is, $\win(n)=\Z^n\cap \Pi_{n-1}$.
We recall that $\Pi_{n-1}$ is a classical polytope in discrete geometry
(see, e.g., Ziegler \cite{Zie95}), 
obtained as the convex hull of ${\bf w}_n$ and its permutations,
see Figure \ref{F_perm}. 
By Stanley \cite{Sta80}, win sequences 
are in bijection with spanning forests $F\subset K_n$. 
(The volume of $\Pi_{n-1}$ is the number of 
spanning trees $T\subset K_n$.) 
See Postnikov \cite{Pos09} for generalizations. 

It is convenient to make a linear shift 
\begin{equation}\label{E_Pi_prime}
\Pi_{n-1}'=\Pi_{n-1}-\frac{n-1}{2}{\bf 1}_n,
\end{equation}
where ${\bf 1}_n=(1,\ldots,1)\in\Z^n$. 
Note that this re-centers the polytope at 
the origin 
${\bf 0}_n=(0,\ldots,0)\in\Z^n$.  
The {\it score sequence} 
\[
{\bf s}(T)={\bf w}(T)-\frac{n-1}{2}{\bf 1}_n,
\] 
associated with the win sequence 
${\bf w}(T)$ of a tournament $T$, 
is 
given by 
\begin{equation}\label{E_sT}
{\bf s}(T)=\sum_{i>j} (w_{ij}-1/2)({\bf e}_i-{\bf e}_j). 
\end{equation}
This shift 
corresponds to awarding a $\pm1/2$ point for each win/loss. 
We let $\score(n)$ denote the set of all possible score sequences.

Although the set $\score(n)$ has a simple, 
geometric description, 
the set $\tour(n,{\bf s})$, of tournaments 
with given score sequence
${\bf s}$, appears to be quite combinatorially complex. 

Kannan, Tetali and Vempala \cite{KTV99} 
investigated simple random walk
as a way of sampling
from $\tour(n,{\bf s})$. However, rapid mixing was proved 
only for ${\bf s}$ sufficiently close to ${\bf 0}_n$. 
McShine \cite{McS00} established rapid mixing in time $O(n^3\log n)$, 
for {\it all} ${\bf s}\in\score(n)$, 
by an elegant application of 
Bubley and Dyer's \cite{BD97} method of path coupling, which 
was relatively new at the time.
See Section \ref{S_tmix} below for an overview. 
We note that Sarkar \cite{Sar20} has shown that 
mixing takes $\Omega(n^3)$ for some sequences. 

More specifically, in \cite{KTV99,McS00}, the random walks are
on the {\it interchange graphs} $\ig(n,{\bf s})$ introduced by 
Brualdi and Li \cite{BL84}. 
Any two tournaments with the same score sequence have the 
same number of copies of the {\it cyclic triangle} 
$\Delta_c$ (see Figure \ref{F_genD}). 
Moreover, if some $T$ contains a copy $\Delta$ of $\Delta_c$, 
then the tournament $T*\Delta$, obtained by reversing the orientation
of all edges in $\Delta$, has the same score sequence as $T$. 
These observations are the key to exploring $\tour(n,{\bf s})$.
The graph 
$\ig(n,{\bf s})$ has a vertex $v(T)$ for each $T\in \tour(n,{\bf s})$
and two $v(T_1),v(T_2)$ are neighbors if $T_2=T_1*\Delta$
for some copy $\Delta\subset T_1$ of $\Delta_c$. 
It can be shown that $\ig(n,{\bf s})$ is connected. 
In this sense, $\Delta_c$ {\it generates} the set 
$\tour(n,{\bf s})$. 

The permutahedron $\Pi_{n-1}$ 
is related with the standard root system 
of type $A_{n-1}$, as the symmetric group 
$S_n$ is the Weyl group of type $A_{n-1}$. 
We refer to, e.g., the standard text by 
Humphreys \cite{Hum90} for background 
on root systems.  
We recall that Killing \cite{Kil90} and 
Cartan \cite{Car96} classified 
all (irreducible, crystallographic) root systems 
(up to isomorphism) as 
the infinite families 
$A_{n-1}$, $B_n$, $C_n$ and $D_n$ 
and the finite exceptional types 
$E_6$, $E_7$, $E_8$, $F_4$ and $G_2$. 

Coxeter permutahedra $\Pi_\Phi$, recently studied
by Ardila, Castillo, Eur and Postnikov \cite{ACEP20}, 
are obtained by 
replacing the role of $S_n$ in the definition of $\Pi_{n-1}$ with the 
Weyl group $W_\Phi$ of a root system $\Phi$. 
See, e.g., Figure \ref{F_permC}. 

\begin{figure}[h!]
\centering
\includegraphics[scale=0.9]{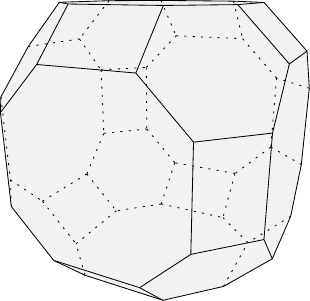}
\caption{The Coxeter permutahedron of type $C_3$.
}
\label{F_permC}
\end{figure}

The previous works in this series 
\cite{KS23,KMP23} studied the connection between 
the polytopes 
$\Pi_\Phi$ and {\it Coxeter tournaments,} 
which are related
to orientations of signed graphs, as developed by 
Zaslavsky  \cite{Zas81,Zas82,Zas91}.
As mentioned above, these tournaments 
involve collaborative and solitaire games, 
as well as the usual competitive games 
in classical graph tournaments.

%%%%%%%%%%%%%%%%%%%%%%%%%%%%%%%%%%%%%%
%%%%%%%%%%%%%%%%%%%%%%%%%%%%%%%%%%%%%%
%%%%%%%%%%%%%%%%%%%%%%%%%%%%%%%%%%%%%%
%%%%%%%%%%%%%%%%%%%%%%%%%%%%%%%%%%%%%%
%%%%%%%%%%%%%%%%%%%%%%%%%%%%%%%%%%%%%%
%%%%%%%%%%%%%%%%%%%%%%%%%%%%%%%%%%%%%%
\subsection{Purpose}
\label{S_purpose}

In this work 
(see Theorem \ref{T_Main})
we show that simple random walks
mix rapidly on the {\it Coxeter interchange graphs}
$\ig(\Phi,{\bf s})$. These graphs 
encode the combinatorics of the 
sets $\tour(\Phi,{\bf s})$, of Coxeter tournaments
with a given score sequence ${\bf s}$, 
and give structural information 
about the fibers of the Coxeter permutahedra $\Pi_\Phi$. 
We focus on the non-standard types $\Phi=B_n$, $C_n$ and $D_n$.  

We also show (see Theorem \ref{T_MainInt}) 
that all Coxeter interchange graphs are connected
and we bound their diameter. 
In constructing our random walk couplings, 
we uncover various other fine, structural properties of the 
graphs $\ig(\Phi,{\bf s})$, and hence the sets $\tour(\Phi,{\bf s})$, 
which might be of independent 
(algebraic, geometric, etc.) interest.

%%%%%%%%%%%%%%%%%%%%%%%%%%%%%%%%%%%%%%
%%%%%%%%%%%%%%%%%%%%%%%%%%%%%%%%%%%%%%
%%%%%%%%%%%%%%%%%%%%%%%%%%%%%%%%%%%%%%
%%%%%%%%%%%%%%%%%%%%%%%%%%%%%%%%%%%%%%
%%%%%%%%%%%%%%%%%%%%%%%%%%%%%%%%%%%%%%
%%%%%%%%%%%%%%%%%%%%%%%%%%%%%%%%%%%%%%
\subsection{Discusssion}
\label{S_discussion}

Path coupling is a powerful method for 
establishing rapid mixing 
(see Section \ref{S_tmix}). 
As already mentioned, path coupling was used in \cite{McS00}. 
We will also use this method, 
however, the application 
in the Coxeter setting 
is significantly more delicate. 

As discussed above, the interchange graphs in type $A_{n-1}$
are generated by a single neutral tournament, namely, 
the cyclic triangle $\Delta_c$. 
On the other hand, in the Coxeter setting, 
there are a number of other
generators which play a role 
(see Figures \ref{F_genB}, \ref{F_genC} and \ref{F_genD}). 
A fascinating interplay arises, as these  
generators can interact in a variety of interesting ways. 
As such, the Coxeter interchange graphs are much richer
in complexity. Likewise, the analysis of random walks on these
structures is more involved. 

The types increase in difficulty in order $A_{n-1}$, $D_n$, 
$B_n$, $C_n$. Type $C_n$ is especially challenging, 
due to the presence of loops in some of the generators, 
which we call {\it clovers} (see Figure \ref{F_genC}). 
These generators correspond to double edges 
in an interchange graph. 
In particular, a special type of structure, 
which we call a {\it crystal} (see
Figure \ref{F_crystal}), can appear in type $C_n$ interchange graphs.
The crystal arises in $C_3$ when 
 ${\bf s}=(2,1,1)$. Crystals can also be found as subgraphs in larger
type $C_n$ interchange graphs, for instance, in the {\it snare drum}
in Figure \ref{F_drum}, when ${\bf s}=(-1,0,1)$. 
In other examples, such as the {\it tambourine} in Figure \ref{F_tamb}, 
when ${\bf s}=(0,0,0)$ is the center of the polytope, 
there are 
no crystals, but interesting structure nonetheless. 

\begin{figure}[h!]
\centering
\includegraphics[scale=1.15]{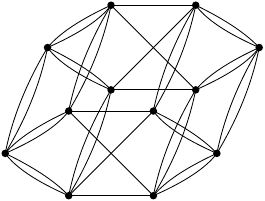}
\caption{The snare drum
interchange graph 
$\ig(C_3,{\bf s})$, when ${\bf s}=(-1,0,1)$,  
is the Cartesian product of a double edge
and the crystal (see Figure \ref{F_crystal}).  
}
\label{F_drum}
\end{figure}

The presence of crystals in type $C_n$ 
interchange graphs 
leads to two main issues.
The first is in extending certain natural couplings 
on various small subgraphs 
(the {\it extended networks} discussed in 
Sections \ref{S_class_net} and \ref{S_ex_net})
to a unified coupling on the entire interchange graph. 
To overcome this difficulty, 
we will prove a number of detailed combinatorial properties of the 
Coxeter interchange graphs. We classify the types of 
subgraphs (see Figures \ref{F_diamonds1}, 
\ref{F_diamonds2} and \ref{F_crystal}), 
which together form the full graph,
and study the ways in which they can intersect. 
For example, one crucial property (see Lemma \ref{L_crystals2})
is that any two crystals can share at most one single edge. 
Without this property, it seems that a path coupling argument
would not be possible. 

The second issue caused by crystals is in obtaining a 
``contractive'' (see Section \ref{S_tmix}) coupling. 
In applying path coupling, we will need
to re-weight the graph metric in 
a specific way, which accounts for 
the occurrence of crystals. 
Loosely speaking, the choice of weights
is related to the fact that, in our random walk couplings, 
crystals work like ``switches,'' that convert
single edges to double edges, and vice versa. 

The overall coupling used to establish rapid mixing 
in the Coxeter setting 
is quite elaborate.
See, e.g., Figures \ref{F_coup_C1a}, 
\ref{F_coup_C1b} and \ref{F_coup_C2} below.
The classical type $A_{n-1}$ result 
\cite{KTV99,McS00} is a special case of 
the argument depicted in 
Figure \ref{F_coup_BD}.

\begin{figure}[h!]
\centering
\includegraphics[scale=1.15]{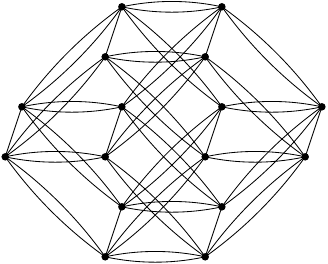}
\caption{The tambourine interchange graph 
$\ig(C_3,{\bf s})$, when ${\bf s}=(0,0,0)$ is 
the center of the type $C_3$ permutahedron. 
This graph is the Cartesian product of a single edge
and the cube of double edges. 
}
\label{F_tamb}
\end{figure}

%%%%%%%%%%%%%%%%%%%%%%%%%%%%%%%%%%%%%%
%%%%%%%%%%%%%%%%%%%%%%%%%%%%%%%%%%%%%%
%%%%%%%%%%%%%%%%%%%%%%%%%%%%%%%%%%%%%%
%%%%%%%%%%%%%%%%%%%%%%%%%%%%%%%%%%%%%%
%%%%%%%%%%%%%%%%%%%%%%%%%%%%%%%%%%%%%%
%%%%%%%%%%%%%%%%%%%%%%%%%%%%%%%%%%%%%%
\section{Background}
\label{S_back}

We refer to \cite{ACEP20}, 
Humphreys \cite{Hum90}, Zaslavsky  \cite{Zas81,Zas82,Zas91}, 
and the previous works in this series
\cite{KS23,KMP23} for a detailed background on root systems, 
signed graphs
and their connections to discrete geometry. In this section, we 
will only recall what is used in the current work.

%%%%%%%%%%%%%%%%%%%%%%%%%%%%%%%%%%%%%%
%%%%%%%%%%%%%%%%%%%%%%%%%%%%%%%%%%%%%%
%%%%%%%%%%%%%%%%%%%%%%%%%%%%%%%%%%%%%%
%%%%%%%%%%%%%%%%%%%%%%%%%%%%%%%%%%%%%%
%%%%%%%%%%%%%%%%%%%%%%%%%%%%%%%%%%%%%%
%%%%%%%%%%%%%%%%%%%%%%%%%%%%%%%%%%%%%%
\subsection{Coxeter tournaments}

A {\it signed graph} $\cS$ on $[n]=\{1,2,\ldots,n\}$ has a set of signed edges
$E(\cS)$. The four possible types of edges are:  
\begin{itemize}[ ]
\item {\it negative edges} $e^-_{ij}$ between two vertices $i$ and $j$, 
\item {\it positive edges} $e^+_{ij}$ between two vertices $i$ and $j$, 
\item {\it half edges} $e^h_i$ with only one vertex $i$, and 
\item {\it loops} $e^\ell_i$ at a vertex $i$. 
\end{itemize}
We note that 
classical graphs $G$ correspond to 
signed graphs $\cS$ with only negative edges. 

In this work, we focus on the 
{\it complete signed graphs} $\cK_\Phi$ 
of types $\Phi=B_n$, $C_n$ and $D_n$. 
These signed graphs contain 
all possible negative and positive edges $e_{ij}^\pm$. 
In type $B_n$ (resp.\ $C_n$), 
all possible half edges $e_i^h$ (resp.\ loops $e_i^\ell$) 
are also included.
We call a signed graph $\cS$ a {\it $\Phi$-graph}
if $\cS\subset \cK_\Phi$. 

Most of the results in the literature on 
classical (type $A_{n-1}$)
graph tournaments restricts to the 
case that $G$ is the complete graph $K_n$. 
We note that the signed graph $\cK_{A_{n-1}}$
with all possible negative edges 
(and no other types of signed edges)
corresponds to 
the classical complete graph $K_n$.

A {\it Coxeter tournament} $\cT$ 
on a signed graph $\cS$ is an orientation of $\cS$. 
When $\cS$ is unspecified, our default assumption will be 
that $\cS=\cK_\Phi$.
More formally, $\cT=(w_e:e\in E(\cS))$, with all $w_e\in\{0,1\}$.  
We think of each $e\in E(\cS)$ as a {\it game}, and $w_e$
as indicating its outcome. 
(We think of $E(\cS)$ as having 
a natural ordering, so that 
$(w_e:e\in E(\cS))$ holds all 
information about the orientation of $\cS$ under
$\cT$. However, we could, somewhat pedantically, instead write 
$\cT=\{(e,w_e):e\in E(\cS)\}$.)

The {\it score sequence} is given by, cf.\ \eqref{E_sT}, 
\[
{\bf s}(\cT)
=\sum_{e\in E(\cS)} (w_e-1/2){\bf e},
\]
where 
${\bf e}$ is the vector corresponding to 
the signed edge $e$, given by 
${\bf e}_{ij}^\pm={\bf e}_i\pm {\bf e}_j$, 
${\bf e}_{i}^h={\bf e}_i$ and ${\bf e}_{i}^\ell=2{\bf e}_i$. 
In other words: 
\begin{itemize}[ ]
\item negative edges $e_{ij}^-$ are {\it competitive games} in which 
one of $i,j$ wins and the other loses a $1/2$ point, 
\item positive edges $e_{ij}^+$ are {\it collaborative games} in which 
$i,j$ both win or lose a $1/2$ point, 
\item half edges $e_i^h$ are {\it (half edge) solitaire games} in which 
$i$ wins or loses a $1/2$ point, and 
\item loops $e_i^\ell$ are {\it (loop) solitaire games} in which 
$i$ wins or loses $1$ point. 
\end{itemize}

If ${\bf s}(\cT)={\bf 0}_n$ we say that $\cT$ is {\it neutral}. 

We let ${\bf s}_\Phi$ denote the {\it standard score sequence} 
corresponding
to the Coxeter tournament in which all 
$w_e=1$. We note that, in some contexts, ${\bf s}_\Phi$ is called the
{\it Weyl vector}. It is also the sum of the {\it fundamental weights}
of the root system $\Phi$. 
See, e.g., \cite{Hum90,Hal15} for more details. 

As discussed in \cite{ACEP20}, the 
 {\it Coxeter $\Phi$-permutahedron} $\Pi_\Phi$ is 
the convex hull of the orbit of ${\bf s}_\Phi$ 
under the Weyl group $W_\Phi$
of type $\Phi$. Thus ${\bf s}_\Phi$ is a 
distinguished vertex of $\Pi_\Phi$. 
Note that the symmetric group $S_n$ is the Weyl group 
of standard type $\Phi=A_{n-1}$
and the Weyl vector is ${\bf s}_n={\bf w}_n-\frac{n-1}{2}{\bf 1}_n$, 
so $\Pi_{n-1}'$
(see \eqref{E_Pi_prime} above) 
is the {\it $\Phi$-permutahedron} of standard type $\Phi=A_{n-1}$. 

In \cite{KS23}, we showed that 
$\Pi_\Phi$ is precisely the set of 
all possible {\it mean} score
sequences of {\it random}  Coxeter tournaments, 
thereby establishing a Coxeter analogue of 
a classical result of Moon \cite{Moo63}. 
The next work in this series \cite{KMP23} focused on 
deterministic Coxeter tournaments. 
The set $\score(\Phi)$ of all 
score sequences of Coxeter tournaments
was classified, generalizing 
the classical result of Landau \cite{Lan53}
discussed above. 

The precise characterization of $\score(\Phi)$ 
is somewhat technical, involving a 
certain weak sub-majorization condition 
and additional parity conditions
in types $C_n$ and $D_n$. 
The proof is constructive, in that it shows 
how to build a Coxeter tournament 
with any given score sequence. 
See \cite[Theorem 4]{KMP23} for more details.

%%%%%%%%%%%%%%%%%%%%%%%%%%%%%%%%%%%%%%
%%%%%%%%%%%%%%%%%%%%%%%%%%%%%%%%%%%%%%
%%%%%%%%%%%%%%%%%%%%%%%%%%%%%%%%%%%%%%
%%%%%%%%%%%%%%%%%%%%%%%%%%%%%%%%%%%%%%
%%%%%%%%%%%%%%%%%%%%%%%%%%%%%%%%%%%%%%
%%%%%%%%%%%%%%%%%%%%%%%%%%%%%%%%%%%%%%
\subsection{Interchange graphs}
\label{S_IntGr}

The set $\tour(\Phi,{\bf s})$ of all 
Coxeter tournaments on $\cK_\Phi$ 
with given score sequence ${\bf s}$
was also investigated in \cite{KMP23}. 
Coxeter analogues of the 
interchange graphs $\ig(n,{\bf s})$ 
discussed above were introduced. 
Recall that the sets $\tour(n,{\bf s})$ 
are generated by the cyclic
triangle $\Delta_c$. 
In the Coxeter setting, 
there are additional generators. 

In all types $B_n$, $C_n$ and $D_n$, in addition to 
$\Delta_c$, we also require a {\it balanced triangle} 
$\Delta_b$. In type $B_n$, 
there are also three 
{\it neutral pairs}
$\Omega_1$, $\Omega_2$ and $\Omega_3$. 
On the other hand, in type $C_n$, there are also 
two {\it neutral clovers}
$\Theta_1$ and $\Theta_2$.
See Figures \ref{F_genD}, \ref{F_genB} and \ref{F_genC}.  
In figures depicting Coxeter tournaments, we will draw:  
\begin{itemize}[]
\item competitive games as edges directed away from their winner, 
\item collaborative games as solid/dotted lines if won/lost,
\item half edge solitaire games as half edges 
directed away/toward from their (only) endpoint if won/lost, and 
\item loop solitaire games as solid/dotted loops if won/lost. 
\end{itemize}

\begin{figure}[h!]
\centering
\includegraphics[scale=1.25]{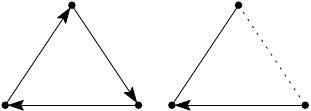}
\caption{The cyclic and balanced
triangles 
$\Delta_c$ and $\Delta_b$ are generators
in all types $B_n$, $C_n$ and $D_n$.}
\label{F_genD}
\end{figure}

\begin{figure}[h!]
\centering
\includegraphics[scale=1.25]{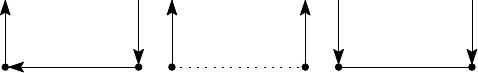}
\caption{The neutral pairs
$\Omega_1$, $\Omega_2$ and $\Omega_3$
are additional generators in type $B_n$. 
}
\label{F_genB}
\end{figure}

\begin{figure}[h!]
\centering
\includegraphics[scale=1.25]{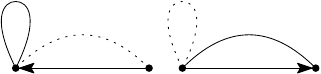}
\caption{The neutral clovers 
$\Theta_1$ and $\Theta_2$
are additional generators in type $C_n$. 
}
\label{F_genC}
\end{figure}

The {\it reversal} $\cT^*$ of a 
Coxeter tournament $\cT=(w_e:e\in E(\cS))$
on $\cS$ is obtained by reversing the outcome of all games
in $\cT$. That is, $\cT^*=(w_e^*:e\in E(\cS))$, where
$w_e^*=1-w_e$. If $\cX\subset\cT$, we let $\cT*\cX$ denote the 
Coxeter tournament obtained from $\cT$ by reversing the outcome
of all games in $\cX$. 
In particular, $\cT^*=\cT*\cT$.

The {\it Coxeter interchange graph} $\ig(\Phi,{\bf s})$ 
has a vertex 
$v(\cT)$ for each $\cT\in \tour(\Phi,{\bf s})$. 
Vertices 
$v(\cT_1),v(\cT_2)$ are neighbors if 
$\cT_2=\cT_1*\cG$, for some copy 
$\cG\subset\cT_1$ of a type $\Phi$
generator. If $\cG$ is a neutral clover, we add a double edge, 
and otherwise we add a single edge. 

For instance, the ``snare drum,'' in Figure \ref{F_drum} above,
is $\ig(C_3,{\bf s})$ when ${\bf s}=(-1,0,1)$. 

The decision to represent clovers as double edges 
might seem arbitrary at first sight, however, 
there is a good reason. 
As it turns out, rather miraculously, 
this adjustment makes the type 
$C_n$ interchange graphs 
degree regular. 
Furthermore, the degree of 
$\ig(\Phi,{\bf s})$ is related to distances in  
$\Pi_\Phi$ in the following way. 
Let $\| {\bf x}\|^2=\sum_i x_i^2$ 
denote the squared length of ${\bf x}\in\R^n$. 

Recall that $\score(\Phi)$ is the set of all possible score
sequences of Coxeter tournaments on $\cK_\Phi$. 
As discussed above, this set is classified in \cite{KMP23}. 

\begin{theorem}[\hspace{1sp}\cite{KMP23}]
\label{T_deg}
Let $\Phi=B_n$, $C_n$ or $D_n$. 
Fix any ${\bf s}\in \score(\Phi)$. 
Then the Coxeter interchange graph 
$\ig(\Phi,{\bf s})$ is regular, with degree
given by 
\[
d(\Phi,{\bf s})
=\frac{\|{\bf s}_\Phi\|^2-\| {\bf s} \|^2}{2},
\]
where ${\bf s}_\Phi$ is the standard score sequence.
\end{theorem}

In particular, $d(\Phi,{\bf s})= O(n^3)$. 

We observe that, as ${\bf s}$ moves closer to the 
center ${\bf 0}_n$ of the polytope $\Pi_\Phi$, the degree 
$d(\Phi,{\bf s})$ of $\ig(\Phi,{\bf s})$ increases. 
This is in line with the intuition that Coxeter tournaments
with ${\bf s}$ closer to ${\bf 0}_n$ (i.e., closer to being neutral)
should contain 
more copies of the (neutral) generators. 

Let us note that ${\bf s}$
with $\|{\bf s}\|^2=\|{\bf s}_\Phi\|^2$ 
are precisely the vertices of 
$\Pi_\Phi$. For such ${\bf s}$, 
we have $d(\Phi,{\bf s})=0$, 
in line with the fact there is a 
unique Coxeter tournaments
with score sequence ${\bf s}$. 
Indeed, such a tournament is transitive, 
in the sense that it 
contains no copy of a neutral generator, 
and so its 
interchange graph is single isolated vertex. 

In \cite{KMP23}, we observed that 
such a result also holds 
for graph tournaments, in relation 
to the standard permutahedron,
yielding a geometric interpretation of the classical result
(see, e.g., Moon \cite{Moo68}) 
that any two tournaments with the same win/score sequence
have the same number of cyclic triangles.  

In closing, let us emphasize the 
the neutral generators in 
Figures \ref{F_genD}, \ref{F_genB} and \ref{F_genC}
were identified in \cite{KMP23}. 
Theorem \ref{T_deg}, proved therein, identifies the degree
of the interchange graphs. 
However, in the current work, we will show 
(see Theorem \ref{T_MainInt} below) that the 
interchange graphs are connected. 
It is this result that justifies calling these structures
``generators,'' in the sense that the entire space
$\tour(\Phi,{\bf s})$ 
is obtained by iteratively reversing copies of these
specific neutral structures.

%%%%%%%%%%%%%%%%%%%%%%%%%%%%%%%%%%%%%%
%%%%%%%%%%%%%%%%%%%%%%%%%%%%%%%%%%%%%%
%%%%%%%%%%%%%%%%%%%%%%%%%%%%%%%%%%%%%%
%%%%%%%%%%%%%%%%%%%%%%%%%%%%%%%%%%%%%%
%%%%%%%%%%%%%%%%%%%%%%%%%%%%%%%%%%%%%%
%%%%%%%%%%%%%%%%%%%%%%%%%%%%%%%%%%%%%%
\subsection{Path coupling}
\label{S_tmix}

We recall that an 
aperiodic, irreducible 
discrete-time Markov chain $(X_n)$
on a finite state space $\Omega$
has a unique {\it equilibrium}
$\pi$ on $S$ such that, for all $x,y\in \Omega$, 
\[
p_n(x,y)=\P(X_n=y|X_0=x)\to \pi(y),
\]
as $n\to\infty$. 
We note that $\pi(y)$ is the asymptotic proportion of time
spent at state $y\in \Omega$. 
The maximal total variation distance 
from $\pi$ by time $n$, 
\[
\tau (n) 
= \max_{x\in \Omega}
\frac{1}{2}\sum_{y\in S}|p_n(x,y)-\pi(y)|, 
\]
is non-increasing.
The {\it mixing time} is defined as 
\[
t_{\rm mix}=\inf \{n\ge0:\tau(n)\le 1/4\}.
\]
A Markov chain is said to be {\it rapidly mixing}
if $t_{\rm mix}$ is bounded by a polynomial
in $\log |\Omega|$. 

Path coupling was introduced by 
Bubley and Dyer \cite{BD97}.
See, e.g., 
Aldous and Fill \cite[Sec.\ 12.1.12]{AF02} or 
Levin, Peres and Wilmer \cite[Sec.\ 14.2]{LPW09}   
for reformulations of the original result 
that are
closer in appearance to that of the following. 

Consider a connected graph $G=(V,E)$. 
The {\it graph distance} $\delta(x,y)$ 
is the minimal number of edges in 
a path between $x$ and $y$. The {\it diameter} is
$D=\max_{x,y\in V} \delta(x,y)$ 
is the maximal length of such a path 
in $G$. 

\begin{definition}\label{D_w}
We say that $G=(V,E)$ is {\it weighted} by $w$ 
if each edge $\{u,v\}\in E$ is assigned some 
{\it weight} $w(u,v)\ge 1$. 
The {\it weighted distance} $w(x,y)$ is the minimal total weight path
between $x$ and $y$. Likewise, $D_w=\max_{x,y\in V} w(x,y)$
is the {\it weighted diameter}. 
\end{definition}

The usual graph distance $\delta$ corresponds to the   
$w$
for which $w(u,v)=1$ for all $\{u,v\}\in E$. 

\begin{theorem}[Path coupling, \cite{BD97}]
\label{T_PC}
Consider a Markov chain $(X_n)$
on a connected graph $G=(V,E)$, weighted by $w$. 
Suppose that, for some $\alpha>0$, for each $\{x',x''\}\in E$ 
there is a coupling $(X_1',X_1'')$ with $(X_0',X_0'')=(x',x'')$ so that 
\[
\E[w(X_1',X_1'')]\le (1-\alpha)w(x',x'').
\]  
Then $(X_n)$ mixes in time 
$t_{\rm mix}= O(\alpha^{-1}\log D_w)$. 
\end{theorem}

Often this result is applied with $w=\delta$,
and, indeed, this will suffice for us in types $B_n$ and $D_n$. 
In this case, path coupling has the intuitive interpretation
that if the chain is ``contractive'' in expectation, 
then it is rapidly mixing.   

On the other hand, in the more complicated type $C_n$, 
we will select a careful re-weighting $w$
that takes into account some of the more intricate features
in the interchange graphs of this type. 

Finally, note that 
the type $C_n$ interchange graphs  
are, in fact, {\it multi}graphs. 
Specifically,  
some pairs of vertices 
(corresponding to clover reversals) 
are joined by double edges, as in 
Figures \ref{F_drum} and \ref{F_tamb} above. 
This is for technical convenience, 
as it makes the graph regular, and thereby 
our coupling procedure easier to explain. 
We note that 
Theorem \ref{T_PC}
still applies, since a Markov chain 
$(Y_n)$ on a multigraph $M$ 
with some double edges is 
equivalent to the Markov chain $(X_n)$ 
on the graph $G$, obtained by 
collapsing each double edge in $M$ into a single edge, 
and combining the two associated edge crossing probabilities.

%%%%%%%%%%%%%%%%%%%%%%%%%%%%%%%%%%%%%%
%%%%%%%%%%%%%%%%%%%%%%%%%%%%%%%%%%%%%%
%%%%%%%%%%%%%%%%%%%%%%%%%%%%%%%%%%%%%%
%%%%%%%%%%%%%%%%%%%%%%%%%%%%%%%%%%%%%%
%%%%%%%%%%%%%%%%%%%%%%%%%%%%%%%%%%%%%%
%%%%%%%%%%%%%%%%%%%%%%%%%%%%%%%%%%%%%%
\section{Main result}
\label{S_results}

Our main result shows that random walks 
rapidly mix on the Coxeter interchange graphs. 

\begin{theorem}
\label{T_Main}
Let $\Phi=B_n$, $C_n$ or $D_n$. Fix any ${\bf s}\in \score(\Phi)$. 
Then lazy simple random walk $(\cT_n:n\ge0)$
on 
$\ig(\Phi,{\bf s})$ 
mixes in time 
$t_{\rm mix}= O(n^3\log n)$ if $\Phi=B_n$ or $D_n$, 
and in time $t_{\rm mix}=  O(n^4\log n)$ if $\Phi=C_n$.
\end{theorem}

Rapid mixing in type $A_{n-1}$, proved in \cite{KTV99,McS00}, 
follows as a special case of our proof of this result 
in type $D_n$. 

We note that the classification of 
$\score(\Phi)$ in \cite{KS23} is constructive, 
which allows us to initialize the random walk in the first place. 

In fact, we will prove sharper bounds 
(see Theorem \ref{T_MainBD} and \ref{T_MainC} below).  
In types $B_n$ and $D_n$, we will show that 
$t_{\rm mix}= O(d\log n)$, where
$d$ is the degree (see Theorem \ref{T_deg} above) of 
the interchange graph $\ig(\Phi,{\bf s})$. 
The result above follows, since $d=O(n^3)$. 
In type $C_n$, we will show
that $t_{\rm mix}= O(\gamma d\log n)$, where $\gamma$
is a certain quantity 
(see Lemma \ref{L_crystals_ee}) 
satisfying $\gamma\le \min\{d,2n\}$. 
We call $\gamma$ the {\it maximal crystal degree}
of the interchange graph. Roughly speaking, it is
maximal number of crystals, all containing the same 
double edge. This quantity is related to the re-weighting $w$
that we will use in applying Theorem \ref{T_PC}
in type $C_n$. 
We note that re-weighting arguments have been used before, e.g., 
in the work of Wilson \cite{Wil04}.

%%%%%%%%%%%%%%%%%%%%%%%%%%%%%%%%%%%%%%
%%%%%%%%%%%%%%%%%%%%%%%%%%%%%%%%%%%%%%
%%%%%%%%%%%%%%%%%%%%%%%%%%%%%%%%%%%%%%
%%%%%%%%%%%%%%%%%%%%%%%%%%%%%%%%%%%%%%
%%%%%%%%%%%%%%%%%%%%%%%%%%%%%%%%%%%%%%
%%%%%%%%%%%%%%%%%%%%%%%%%%%%%%%%%%%%%%
\section{Connectivity} 
\label{S_DecompRev}

Before proving Theorem \ref{T_Main},  
we will first establish the following combinatorial result, 
giving further structural information (beyond its regularity, given 
by Theorem \ref{T_deg})
about the Coxeter interchange graphs. 

\begin{theorem}
\label{T_MainInt}
Let $\Phi=B_n$, $C_n$ or $D_n$. 
Fix any ${\bf s}\in \score(\Phi$). 
The Coxeter interchange graph $\ig(\Phi,{\bf s})$
is connected and its diameter 
$D= O(n^2)$. 
\end{theorem}

This result is a corollary of 
Lemma \ref{L_rev} (the ``reversing lemma'') 
proved at the end of this section. 
A number of preliminaries are required. First, 
in the next subsection, we will find a way of encoding
Coxeter tournaments as special types of directed graphs. 

%%%%%%%%%%%%%%%%%%%%%%%%%%%%%%%%%%%%%%
%%%%%%%%%%%%%%%%%%%%%%%%%%%%%%%%%%%%%%
%%%%%%%%%%%%%%%%%%%%%%%%%%%%%%%%%%%%%%
%%%%%%%%%%%%%%%%%%%%%%%%%%%%%%%%%%%%%%
%%%%%%%%%%%%%%%%%%%%%%%%%%%%%%%%%%%%%%
%%%%%%%%%%%%%%%%%%%%%%%%%%%%%%%%%%%%%%
\subsection{Z-frames} 
\label{S_Zframes}

Oriented signed graphs 
were studied by Zaslavsky \cite{Zas91}. 
From this point of view (see \cite[Fig.\ 1]{Zas91}), 
each oriented signed edge 
in an oriented signed graph $\cS$ is the union of 
at most two 
directed half edges.
We modify this idea, by adding 
a named endpoint to each half-edge, 
which we call a \emph{match}. 
This will allow us to prove 
certain structural facts using graph theory techniques. 
We call such a structure a \emph{Z-frame}. 

\begin{definition}
A {\it Z-frame} ${\bf Z}$
is a directed, bipartite
multigraph on disjoint sets of 
{\it players} $V$ and {\it matches} $M$, 
such that every match has degree $1$ or~$2$.
\end{definition}

This concept is fairly general, and not all Z-frames
correspond to a Coxeter tournament $\cT$ on some 
$\cS\subset\cK_\Phi$. However, 
each such $\cT$
has a unique representation as a Z-frame 
${\bf Z}(\cT)$. We think of ${\bf Z}(\cT)$ as revealing 
the ``inner directed graph structure''
of $\cT$. 
Players in ${\bf Z}(\cT)$
correspond to vertices in $\cT$. 
Recall that each game in $\cT$ corresponds to an 
oriented signed edge. 
Each such game is associated with a match 
in ${\bf Z}(\cT)$. 

We will think of edges directed away/toward players $v \in V$ 
as positively/negatively charged. (That being said, 
positive/negative edges in a Z-frame should not be confused
with positive/negative edges in a Coxeter tournament.)

We say that $\bf Z$ is {\it neutral} if all players $v\in V$ 
have net zero charge, i.e.,  
${\rm deg}^+(v)-{\rm deg}^-(v)=0$, 
where ${\rm deg}^\pm(v)$ is the number of positive/negative
edges incident to $v$. 
We put ${\rm deg}(v)={\rm deg}^+(v)+{\rm deg}^-(v)$. 
Note that, if $\bf Z$ is {\it neutral} then ${\rm deg}(v)$ is even.

\begin{definition}
Let $\cT$ be a Coxeter tournament 
on a signed graph $\cS$ 
on $[n]$. 
Let ${\bf Z}(\cT)$ be the Z-frame on $V=[n]$ 
and $M = \{m_e\colon e \in \cS\}$, with the following 
directed edges: 
\begin{itemize}[]
\item If a competitive game $e_{ij}^- \in \cS$ 
between players $i$ and $j$
is won (resp.\ lost) 
by $i$ in $\cT$, we include 
two directed edges $i\rightarrow m_{ij}^- \rightarrow j$ 
(resp.\ $i\leftarrow m_{ij}^- \leftarrow j$).

\item If a collaborative game $e_{ij}^+ \in \cS$ 
between players $i$ and $j$
is won (resp.\ lost) in $\cT$, we 
include two directed edges 
$i\rightarrow m_{ij}^+ \leftarrow j$ 
(resp.\ $i\leftarrow m_{ij}^+ \rightarrow j$). 

\item If a half edge solitaire game $e_i^h \in \cS$ 
by player $i$ 
is won (resp.\ lost) in $\cT$, 
we include one directed edge 
$i\rightarrow m_i^h$ (resp.\ $i\leftarrow m_i^h$). 

\item If a loop solitaire game $e_i^\ell \in \cS$ 
by player $i$
is won (resp.\ lost) in $\cT$, we 
include two directed edges 
$i\rightrightarrows m_i^\ell$ 
(resp.\ $i\leftleftarrows m_i^\ell$).
\end{itemize}
\end{definition}

See Figure \ref{F_ZT} 
for an example of a 
Coxeter tournament $\cT$ and its 
corresponding Z-frame
${\bf Z}(\cT)$.

\begin{figure}[h!]
\includegraphics[scale=1.1]{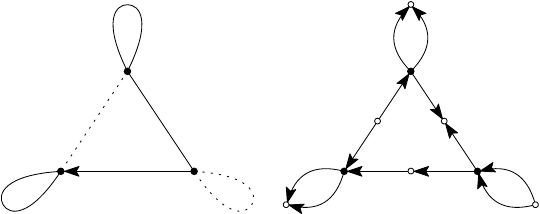}
\caption{A $C_3$-tournament $\cT$ and its Z-frame
${\bf Z}(\cT)$. Players/matches are drawn as black/white dots.}
\label{F_ZT}
\end{figure}

The reason for the above terminology is that 
directed edges in ${\bf Z}(\cT)$ with a positive/negative
charge correspond to positive/negative contributions to 
the score sequence ${\bf s}(\cT)$. 
Note that, in competitive/collaborative games, the two charges
are opposing/aligned 
(regardless of the outcome of the game).
See Figure \ref{F_Zskel}.

 \begin{table}[h!]
 \begin{center}
 \caption{
 Signed edges $e$ in a signed graph $\cS$
 can be oriented in one of two ways 
 $w_e\in\{0,1\}$ by a Coxeter tournament
 $\cT$. The contribution to the score sequence 
 ${\bf s}(\cT)$ and representation in the Z-frame ${\bf Z}(\cT)$
 are given below. 
  }\label{F_Zskel}
 \begin{tabular}{l|l|l|l}

 $e$&$w_e$&${\bf s}$ &${\bf Z}$  \\  \hline
 $e_{ij}^-$&1&$+{\bf e}_{ij}^-/2=+{\bf e}_i/2-{\bf e}_j/2$&$i\rightarrow m_{ij}^- \rightarrow j$\\
 &0&$-{\bf e}_{ij}^-/2=-{\bf e}_i/2+{\bf e}_j/2$&$i\leftarrow m_{ij}^- \leftarrow j$\\\hline
 $e_{ij}^+$&1&$+{\bf e}_{ij}^+/2=+{\bf e}_i/2+{\bf e}_j/2$&$i\rightarrow m_{ij}^+ \leftarrow j$\\
 &0&$-{\bf e}_{ij}^+/2=-{\bf e}_i/2-{\bf e}_j/2$&$i\leftarrow m_{ij}^+ \rightarrow j$\\\hline
 $e_{i}^h$&1&$+{\bf e}^h_i/2=+{\bf e}_i/2$&$i\rightarrow m_i^h$\\
 &0&$-{\bf e}^h_i/2=-{\bf e}_i/2$&$i\leftarrow m_i^h $\\\hline
 $e_{i}^\ell$&1&$+{\bf e}^\ell_i/2=+{\bf e}_i$&$i \rightrightarrows m_i^\ell$\\
 &0&$-{\bf e}^\ell_i/2=-{\bf e}_i$&$i\leftleftarrows m_i^\ell$ 
 \end{tabular}
 \end{center}
 \end{table}

%%%%%%%%%%%%%%%%%%%%%%%%%%%%%%%%%%%%%%
%%%%%%%%%%%%%%%%%%%%%%%%%%%%%%%%%%%%%%
%%%%%%%%%%%%%%%%%%%%%%%%%%%%%%%%%%%%%%
%%%%%%%%%%%%%%%%%%%%%%%%%%%%%%%%%%%%%%
%%%%%%%%%%%%%%%%%%%%%%%%%%%%%%%%%%%%%%
%%%%%%%%%%%%%%%%%%%%%%%%%%%%%%%%%%%%%%
\subsection{Decomposing Z-frames}

Our aim is to decompose neutral Coxeter tournaments 
into irreducible neutral parts.  
In this section, we will do this for Z-frames in general.

Recall that a {\it trail} is a walk in which no 
edge is visited twice.

\begin{definition}
We call a trail ${\bf T}$ of (directed) edges in a Z-frame 
\emph{closed} if it starts and ends at the same vertex,
and otherwise we call it \emph{open}. 
A trail is \emph{neutral} if the two consecutive edges 
at each player $v\in V$ have opposite  charges. 
The {\it length} $\ell$ of a trail is its number of matches.
\end{definition}

Note that 
the edges along a trail 
in a $Z$-frame 
do not repeat, and 
are connected to each other in alternation by a player/match. 
Also note that neutral open trails 
start and end at distinct final matches
(the left and rightmost matches along the trail). 
See Figure \ref{F_trail}.

\begin{figure}[h!]
\includegraphics[scale=1.1]{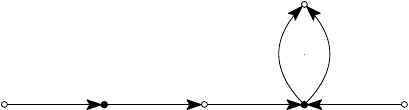}
\caption{An open neutral trail of length 
$\ell=4$.}
\label{F_trail}
\end{figure}

\begin{lemma} \label{L_general_decomp}
Any 
neutral Z-frame 
${\bf Z}$ can be decomposed into an edge-disjoint union 
${\bf Z} = \bigcup_i {\bf T}_i$ of 
closed neutral trails and open neutral trails, 
such that no two open trails have a common final match. 
\end{lemma}

\begin{proof}
Let $v\in V$. Since {\bf Z} is neutral, 
${\rm deg}_+(v)={\rm deg}_-(v)$. 
Therefore we may pair  
each edge directed away from $v$ with an edge 
directed toward $v$. 
Note that any such pair is a neutral trail. 
We let $\cP_v$ denote the set of all such pairs. Then 
$\bigcup_{v\in V}\cP_v$ 
is a decomposition of ${\bf Z}$ into an edge-disjoint union  
of neutral trails.  
Select a decomposition ${\bf Z}=\bigcup_i{\bf T}_i$ 
with the minimal number of neutral trails. 
Open trails in this decomposition cannot 
have a common final match, 
as otherwise they could be concatenated to form a single
longer trail, 
contradicting minimality. 
\end{proof}

Let ${\bf Z}$ be a neutral Z-frame. We say that ${\bf Z}$ is 
{\it reducible} if it contains a 
non-empty neutral ${\bf Z}'\subsetneq {\bf Z}$. 
Otherwise, ${\bf Z}$ is {\it irreducible}. 

Each vertex $v$ in a neutral trail has even degree, 
since $\mathrm{deg}^+(v)=\mathrm{deg}^-(v)$ and 
$\mathrm{deg}(v)=\mathrm{deg}^+(v)+\mathrm{deg}^-(v)$. 
As discussed, vertices $v$ in a neutral ${\bf Z}$
have even $\mathrm{deg}(v)$. Our next result
observes that if ${\bf Z}$ is irreducible, then 
all $\mathrm{deg}(v)\le 4$.

\begin{lemma}\label{L_deg24}
Let ${\bf Z}$ be an irreducible neutral Z-frame.
Then ${\bf Z}$ is a neutral trail and $\mathrm{deg}(v) \in \{0,2,4\}$
for all $v\in V$. 
\end{lemma}

\begin{proof}
By Lemma \ref{L_general_decomp}, ${\bf Z}$ is a neutral trail, 
and so all $\mathrm{deg}(v)$ are even. 
If $v$ is an isolated vertex that plays no solitaire games, then 
$\mathrm{deg}(v)=0$. Otherwise, if ${\bf Z}$ is non-trivial, 
we will argue that $\mathrm{deg}(v)\in\{2,4\}$ are the only possibilities. 
To see this, start at any $v$ along the trail, and then follow the trail. 
Consider the charges of the edges incident to $v$, in the order 
in which they are visited by the trail. 
If the first and second charges are opposing, 
then the trail is complete with  $\mathrm{deg}(v)=2$. 
Otherwise, suppose they are both positive (the 
other case is symmetric). Since 
${\bf Z}$ is neutral, the 3rd charge is negative. 
Finally, consider the 4th charge. 
If it were positive, then 
the sub-trail between the third negative edge incident to $v$ and 
the forth positive edge incident to $v$ would be neutral. 
Therefore, the 4th charge is negative, and so
the trail is complete with $\mathrm{deg}(v)=4$. 
\end{proof}

%%%%%%%%%%%%%%%%%%%%%%%%%%%%%%%%%%%%%%
%%%%%%%%%%%%%%%%%%%%%%%%%%%%%%%%%%%%%%
%%%%%%%%%%%%%%%%%%%%%%%%%%%%%%%%%%%%%%
%%%%%%%%%%%%%%%%%%%%%%%%%%%%%%%%%%%%%%
%%%%%%%%%%%%%%%%%%%%%%%%%%%%%%%%%%%%%%
%%%%%%%%%%%%%%%%%%%%%%%%%%%%%%%%%%%%%%
\subsection{Reversing Coxeter tournaments}
\label{S_DecompZ}

Applying the results of the previous section,  
we obtain the following result for 
Z-frames  ${\bf Z}(\cT)$ of Coxeter tournaments $\cT$. 

\begin{lemma}
Let $\Phi=B_n$, $C_n$ or $D_n$. 
Let ${\bf Z}(\cT)$ be the Z-frame of a neutral 
Coxeter tournament $\cT$ on a signed $\Phi$-graph $\cS\subset\cK_\Phi$. 
Consider a decomposition ${\bf Z}(\cT)=\bigcup_i{\bf T}_i$ 
into neutral trails given by Lemma \ref{L_general_decomp}. 
If $\Phi=C_n$ or $D_n$ then all trails ${\bf T}_i$ are closed. 
If $\Phi=B_n$ then possibly some ${\bf T}_i$ are open. 
\end{lemma}

\begin{proof}
This result follows by noting that if $\cT$ is of type $C_n$ or $D_n$
then all matches in ${\bf Z}(\cT)$ are degree 2. 
Therefore, there are no open neutral trails in the decomposition, 
so each ${\bf T}_i$ is closed.
On the other hand, in type $B_n$ it is possible 
to have open and closed trails, since in this case
possibly some matches are degree 1. 
\end{proof}

Note that a tournament $\cT$ 
on a signed graph $\cS$ is neutral 
if and only if 
its Z-frame ${\bf Z}(\cT)$ is neutral. 
Naturally, we call such a $\cT$ \emph{irreducible} 
if ${\bf Z}(\cT)$ is irreducible, and \emph{reducible} otherwise.

\begin{lemma}\label{L_decomp4}
Suppose that
$\cS$ is a $D_n$-graph in which all players $v\in V$ have degree four. 
Then any neutral tournament $\cT$ on $\cS$ is reducible. 
\end{lemma}

\begin{proof}
The proof is by contradiction. 
Suppose that $\cT$ on $\cS$ is neutral 
and that ${\bf Z}(\cT)$ is irreducible. 
By Lemma \ref{L_general_decomp}, ${\bf Z}(\cT)$ is in fact a single
neutral trail, which for convenience we will denote by $\bf T$. 
Following the consecutive edges of $\bf T$, we can find 
a closed trail ${\bf T}'\subset {\bf T}$ which:  
\begin{itemize}[ ]
\item starts and ends at some player $v$,
\item visits no player $u\neq v$ more than once along the way, 
\item and is neutral everywhere except at $v$.
\end{itemize} 
Let us assume that 
the two directed edges in ${\bf T}'$ incident to $v$
are positive, since the other case is symmetric. 

Consider the extension 
${\bf T}'\cup {\bf T}''$
of ${\bf T}'$, as it departs $v$ via some negatively charged directed edge, 
until it eventually returns to some player $u\in{\bf T}'$ for the first time.
Note that $u\neq v$, as else, since ${\bf T}$ is irreducible, 
it would follow that ${\bf T}={\bf T}'\cup {\bf T}''$, 
and then that there are degree 2 vertices in ${\bf T}$ along ${\bf T}'$. 
Since ${\bf T}'$ is neutral at $u$, the edges in 
${\bf T}'$ incident to $u$ have opposing charges. 
Let ${\bf T}'={\bf T}'_+\cup {\bf T}'_-$, 
where ${\bf T}'_\pm\subset{\bf T}'$ is the trail between $u$ and $v$ 
which includes the positive/negative edge in 
${\bf T}'$ incident to $u$, as in Figure \ref{F_decomp4}.
To conclude, consider the charge of the directed 
edge in ${\bf T}''$ along which $u$ is revisited. 
To obtain the required contradiction, note that if this charge is 
positive (resp.\ negative) 
then ${\bf T}'_-\cup{\bf T}''$ (resp.\ ${\bf T}'_+\cup{\bf T}''$) is a neutral trail. 
\end{proof}

\begin{figure}[h!]
\includegraphics[scale=1.15]{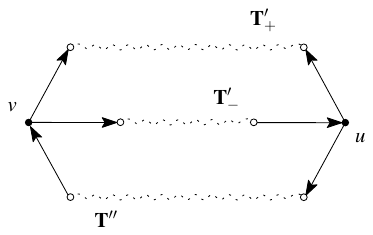}
\caption{
In this instance, the directed edge in 
${\bf T}''$ along which $u$ is revisited is positive, 
so  
${\bf T}'_-\cup{\bf T}''$ is a neutral trail. 
If it were
negative, then ${\bf T}'_+\cup{\bf T}''$ would be neutral.
Curved dotted lines represent the continuation of a trail. 
}
\label{F_decomp4}
\end{figure}

Finally, we turn our attention to the main result of this section. 

\begin{lemma}[Reversing lemma]
\label{L_rev}
Let $\Phi=B_n$, $C_n$ or $D_n$. 
Let $\cT$ be a tournament on the complete signed graph $\cK_\Phi$, 
and let $\cI\subset\cT$ be a neutral sub-tournament. If $\cI$ has 
$\ell\ge3$ games 
then $\cI$ can be reversed 
in a series of at most $\ell-2$ type $\Phi$ generator reversals. 
\end{lemma}

The assumption that $\cT$ is on $\cK_\Phi$ is crucial, since 
not all of the games in the generators 
used to reverse $\cI$ will be in $\cI$ itself. 
However, after the series of reversals, 
only the games in $\cI$ will have been 
reoriented, i.e.,  
all games in $\cT \setminus \cI$ will be restored to 
their initial orientations.

\begin{proof}
Without loss of generality we may assume that $\cI$ is irreducible. 
In this case, by Lemma \ref{L_general_decomp}, 
the Z-frame $\bbZ(\cI)$ is a single neutral trail. 
For simplicity, we will speak of $\cI$ 
as a tournament and trail interchangeably. 

We first address the simplest case of open neutral trails, 
which appears only in type $B_n$. We proceed by induction 
on the length $\ell$. 
The smallest open trails, with $\ell=3$, are the neutral pairs
$\Omega_i$ (as in Figure \ref{F_genB}) themselves, 
which are clearly reversible in a single reversal. 
For longer open neutral trails $\cI$, 
consider any (half edge) solitaire
game $h$ in $\cT$ played by some $u$ in $\cI$, 
which is not an endpoint
of the trail. 
Using $h$, we can first reverse
either the part of the trail that is to the ``left'' of $u$
or else the part which is to the ``right.'' 
We can do this using the inductive hypothesis, 
as one of the parts is neutral and both have length smaller than $\ell$.
Then, using the reversal of this game $h^*$, 
we can reverse the 
other part of the trail in turn.
See Figure \ref{F_rev_trail_open} for an example. 

\begin{figure}[h!]\includegraphics[scale=1.15]{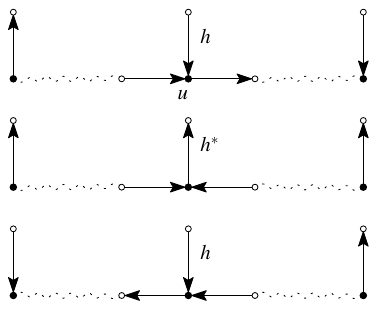}
\caption{In type $B_n$: 
Reversing an open neutral trail of length $\ell>3$, 
using a (half edge) solitaire game $h$ played by some $u$
in the ``middle'' of the trail. 
In this example, from top to bottom, 
we first reverse the ``right'' side of $\cI$ using $h$, and then
the ``left'' side using $h^*$. 
}
\label{F_rev_trail_open}
\end{figure}

Next, we turn to the case of closed neutral trails. 
Recall that closed trails do not have solitaire half edge games, 
so from this point on
we assume that 
$\Phi=C_n$ or $D_n$. If $\cI$ is a generator, then the result clearly holds. 
If it is not, then let $k \ge 3$ be the number of vertices in $\cI$. 
If $k=3$, one can verify directly that the result holds. 
See Figure \ref{F_v3}. 

\begin{figure}[h!]
\includegraphics[scale=1.15]{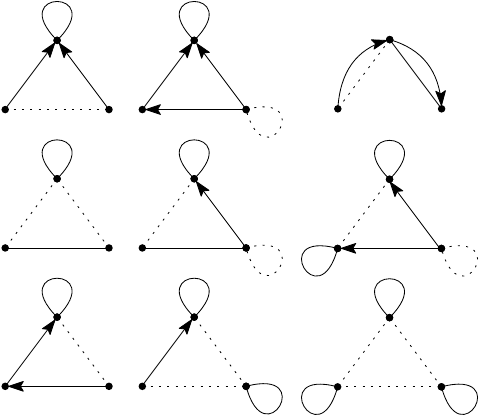}
\caption{
Up to symmetry, 
it suffices to consider the above neutral 
tournaments on three vertices. Note that each of these  
tournaments
with $\ell$ games can be reversed in 
$\ell-2$ steps. 
}
\label{F_v3}
\end{figure}

For $k \ge 4$, we aim to find a pair of vertices 
$i$, $j$ in $\cI$ which do not play a game with each other in  $\cI$.  
Once we find such a pair, the argument is similar to the case
of open trails above; 
the difference being that, in the case of closed trails, we will either first
reverse the ``top/bottom'' (instead of the ``left/right'') of the trail, 
and then the other side
in turn, as in Figure \ref{F_rev_trail_closed}. 

\begin{figure}[h!]\includegraphics[scale=1.15]{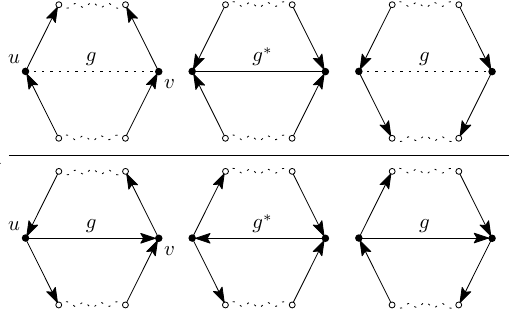}
\caption{
Reversing $\cI$ when its Z-frame is a closed neutral trail. 
{\it Above:}  If the charges of the edges incident to $u$ and $v$
are ``aligned,'' we use the 
collaborative game $g$ between $u,v$
to reverse $\cI$. 
{\it Below:} 
Otherwise, if they are ``unaligned,''
we use the competitive 
game between $u$ and $v$.
}
\label{F_rev_trail_closed}
\end{figure}

To this end, suppose that every pair of 
vertices in $\cI$ plays at least one game
with each other. As $k\ge 4$, this implies that the degree of 
every player in $\cI$ is at least $3$. 
Moreover, if some $v$ in $\cI$ plays a loop game, then 
$\deg(v) \ge 5$. 
However, as $\cI$ is neutral and irreducible, 
this would contradict Lemma \ref{L_deg24}. 
Thus, $\cI$ is a $4$-regular $D_n$-graph. 
But this is also impossible, by Lemma \ref{L_decomp4}.
Therefore, for all $k\ge4$, 
there exists a pair $i,j$ of players 
which do not play a game with each other,
and this concludes the proof. 
\end{proof}

Finally, using the reversing lemma, we will prove the 
main result of this section.

\begin{proof}[Proof of Theorem \ref{T_MainInt}]
Let $\cT,\cT'\in\tour(\Phi,{\bf s})$. 
The distance between $\cT$ and $\cT'$ 
is the smallest number of the generator reversals 
which transforms $\cT$ into $\cT'$. 
The games in the difference 
$\cD=\cT\setminus\cT'$ are precisely those which 
need to be reversed. 
Since ${\bf s}(\cT)={\bf s}(\cT')$, it follows that $\cD$ is neutral.
Therefore, by Lemma \ref{L_rev}, 
there is a path from $v(\cT)$ to $v(\cT')$ in $\ig(\Phi,{\bf s})$
of length $O(n^2)$. Hence 
$\ig(\Phi,{\bf s})$ is connected and its diameter
$D= O(n^2)$. 
\end{proof}

%%%%%%%%%%%%%%%%%%%%%%%%%%%%%%%%%%%%%%
%%%%%%%%%%%%%%%%%%%%%%%%%%%%%%%%%%%%%%
%%%%%%%%%%%%%%%%%%%%%%%%%%%%%%%%%%%%%%
%%%%%%%%%%%%%%%%%%%%%%%%%%%%%%%%%%%%%%
%%%%%%%%%%%%%%%%%%%%%%%%%%%%%%%%%%%%%%
%%%%%%%%%%%%%%%%%%%%%%%%%%%%%%%%%%%%%%
\section{Interchange networks}
\label{S_Networks}

In this section, for ease of exposition, 
we will speak of Coxeter tournaments $\cT$
and their corresponding vertices $v(\cT)$ in $\ig$ interchangeably. 

\begin{definition}
For $\cT_1,\cT_2\in\tour(\Phi,{\bf s})$
at distance two in $\ig(\Phi,{\bf s})$, we define the   
{\it interchange network}
$N(\cT_1,\cT_2)$ to be the 
union over all paths of length two between $\cT_1,\cT_2$. 
\end{definition}

Note that each path of length two 
between such $\cT_1,\cT_2$
corresponds to a way of 
reversing the {\it difference} $\cD=\cT_1\setminus\cT_2$
between $\cT_1$ and $\cT_2$.

%%%%%%%%%%%%%%%%%%%%%%%%%%%%%%%%%%%%%%
%%%%%%%%%%%%%%%%%%%%%%%%%%%%%%%%%%%%%%
%%%%%%%%%%%%%%%%%%%%%%%%%%%%%%%%%%%%%%
%%%%%%%%%%%%%%%%%%%%%%%%%%%%%%%%%%%%%%
%%%%%%%%%%%%%%%%%%%%%%%%%%%%%%%%%%%%%%
%%%%%%%%%%%%%%%%%%%%%%%%%%%%%%%%%%%%%%
\subsection{Classifying networks}
\label{S_class_net}

In this section, we will
classify the possibilities for $(N,\cD)$. 
As we will see, this is the key to applying path coupling
(Theorem \ref{T_PC} above)
in Section \ref{S_Coupling} below. 
In the classical case of type $A_{n-1}$ there is only one possibility
for $N$ (a ``single diamond''), and this is the reason why
such a simple contractive coupling 
(as in Figure \ref{F_coup_BD} below) 
is possible.  As it turns out, this continues
to hold in types $B_n$ and $D_n$, 
but the underlying reasons are more complicated. 
Type $C_n$, on the other hand, 
is significantly more complex, as 
then the structure of $N$ can take various other forms. 

It can be seen that any two distinct generators $\cG_1\neq \cG_2$
are either {\it disjoint} $\cG_1\cap \cG_2=\emptyset$ or else
have exactly one game $g$ in common $\cG_1\cap \cG_2=\{g\}$. 
In this case, we say that $\cG_1,\cG_2$ are {\it adjacent}. 

Note that if a path of length two from $\cT_1$ to $\cT_2$
passes through some $\cT_{12}$, then there are two generators 
$\cG_1,\cG_2\subset\cT_{12}$ such that $\cT_i=\cT_{12}*\cG_i$,
for $i\in\{1,2\}$.  
In this way, every such path of length two from 
$\cT_1$ to $\cT_2$ is determined by a {\it midpoint} $\cT_{12}$
and a pair of generators $\cG_1,\cG_2\subset\cT_{12}$. 

There are 
three possible networks when $\cG_1,\cG_2$ are disjoint. 
We call these the {\it single}, {\it double}
and {\it quadruple diamonds}. 
See Figure \ref{F_diamonds1}.
Recall that double edges in $\ig$ correspond to 
neutral clover reversals. All other types of reversals
(neutral triangles and pairs) are represented as single edges.

\begin{figure}[h!]
\includegraphics[scale=1.25]{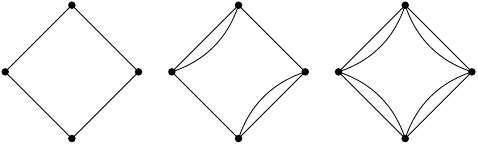}
\caption{
{\it Left to right:} 
The single, double and quadruple diamond
interchange networks.}
\label{F_diamonds1}
\end{figure}

The following result classifies the types of 
networks $N(\cT_1,\cT_2)$ when 
$\cG_1,\cG_2$ are disjoint. 

\begin{lemma}
\label{L_NclassDis}
Suppose that there is a path of length two between $\cT_1,\cT_2$  
in $\ig(\Phi,{\bf s})$ that passes through midpoint $\cT_{12}$,
with associated generators $\cG_i\subset \cT_{12}$ 
such that $\cT_i=\cT_{12}*\cG_i$. 
Suppose that $\cG_1,\cG_2$ are disjoint. 
Then if exactly zero, one or two of the $\cG_i$
are clovers then the 
network $N(\cT_1,\cT_2)$ is a 
single, double or quadruple diamond, 
respectively. 
\end{lemma}

\begin{proof}
Clearly, there are exactly two paths from $\cT_1$ to $\cT_2$. 
These paths correspond to reversing the disjoint generators
$\cG_1^*,\cG_2\subset\cT_1$ in series, in one of the two 
possible orders. See Figure \ref{F_dis}.
\end{proof}

\begin{figure}[h!]
\includegraphics[scale=1.15]{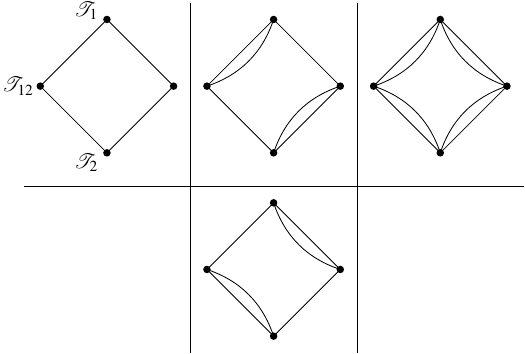}
\caption{
{\it Left to right:} $N(\cT_1,\cT_2)$
when exactly zero, one or two of the 
disjoint $\cG_1,\cG_2$ is a clover (single, double and quadruple diamonds). 
In each cell, the leftmost vertices are $\cT_1$, $\cT_{12}$ and $\cT_2$, 
as in the first cell. 
}
\label{F_dis}
\end{figure}

The case that $\cG_1,\cG_2$ are adjacent is more involved. 
It is useful to note that, in this case, 
the difference $\cD=\cT_1\setminus\cT_2$ is a neutral tournament 
with exactly four games on either three or four vertices. 
Even so, there are a number of cases to consider, 
and the key to a concise argument is grouping symmetric cases together. 
For a Coxeter tournament $\cT$, we define its {\it projection graph}
$\pi(\cT)$ to be the graph obtained by changing each: 
\begin{itemize}[ ]
\item oriented negative/positive edge 
(i.e., competitive/collaborative game) 
into an undirected edge, 
\item oriented half edge 
(i.e., half edge solitaire game) 
into an undirected half edge, 
\item oriented loop 
(i.e., loop solitaire game) 
into an undirected loop.
\end{itemize}

There are a number of ways that two 
generators $\cG_1,\cG_2$ can be adjacent.
However, 
there are only four possibilities for 
their projected difference
$\pi(\cD)$. 
We call these the {\it square}, {\it tent}, {\it fork}
and {\it hanger}, as in Figure \ref{F_Ds}. 
The following observation is essentially self-evident, 
and can be 
verified by an elementary case analysis. 
We omit the proof. 

\begin{figure}[h!]
\includegraphics[scale=1.25]{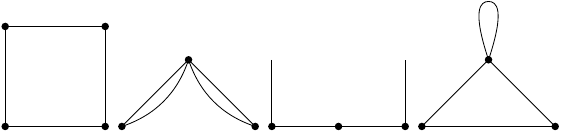}
\caption{
{\it Left to right:} 
The square, tent, fork, and hanger. 
}
\label{F_Ds}
\end{figure}

\begin{lemma}
\label{L_PiclassAdj}
Assume the same set up as Lemma \ref{L_NclassDis}, except that 
instead $\cG_1,\cG_2$ are adjacent. 
Let $\cD=\cT_1\setminus\cT_2$. Then: 
\begin{enumerate}
\item If $\cG_1,\cG_2$ are neutral triangles on four/three vertices
then $\pi(\cD)$ is a square/tent. 
\item If $\cG_1,\cG_2$ are neutral clovers 
then $\pi(\cD)$ is a tent. 
\item If $\cG_1$ is a neutral pair and $\cG_2$ is a neutral 
pair or triangle then $\pi(\cD)$ is a fork. 
\item If $\cG_1$ is a neutral clover and $\cG_2$ is a neutral 
triangle then $\pi(\cD)$ is a hanger. 
\end{enumerate}
\end{lemma}

In addition to the single and double diamond networks 
in Figure \ref{F_diamonds1}, there are two additional networks
that can occur when $\cG_1,\cG_2$ are adjacent.
We call these the {\it split} and {\it heavy diamonds},
see Figure \ref{F_diamonds2}.

\begin{figure}[h!]
\includegraphics[scale=1.25]{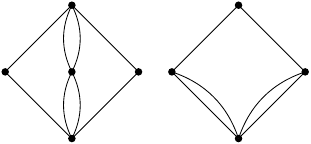}
\caption{
{\it Left to right:} The split and heavy diamond 
interchange networks.}
\label{F_diamonds2}
\end{figure}

The following result classifies the types of 
networks $N(\cT_1,\cT_2)$ when 
$\cG_1,\cG_2$ are adjacent. 

\begin{lemma}
\label{L_NclassAdj}
Assume the same set up as Lemma \ref{L_PiclassAdj} 
(with $\cG_1,\cG_2$ are adjacent). 
Let $\cD=\cT_1\setminus\cT_2$ 
and 
$N=N(\cT_1,\cT_2)$. 
Then: 
\begin{enumerate}
\item If $\Phi=B_n$ or $D_n$, then 
$N$ is a 
single diamond. 
\item If $\Phi=C_n$ and $\pi(\cD)$ is a square, then 
$N$ is a single diamond. 
\item If $\Phi=C_n$ and $\pi(\cD)$ is a tent, then 
$N$ is a split diamond. 
\item If $\Phi=C_n$ and $\pi(\cD)$ is a hanger, then 
$N$ is a double or heavy diamond. 
\end{enumerate}
\end{lemma}

\begin{proof}
{\bf Case 1a.}
We start with the simplest case that $\Phi=B_n$
or $D_n$ and 
$\cG_1,\cG_2$ are adjacent neutral triangles on four vertices,
so that $\pi(\cD)$ is a square. 
Note that, to reverse $\cD$ in two steps, we must reverse
exactly two edges in $\cD$ in each step. As such, no neutral 
pairs will be involved in reversing $\cD$ in two steps. 
There are exactly two ways 
to reverse $\cD$. For each pair of ``antipodal'' vertices
in $\cD$, consider the two games played between the pair. 
Exactly one of the two
games $g$ allows us to reverse the games in $\cD$ 
on one ``side'' of $g$. 
Then, in turn, we can use $g^*$ to reverse the
other two games in $\cD$. 
See Figure \ref{F_rev_square} for all the possible cases of $\cD$. 

\begin{figure}[h!]
\includegraphics[scale=1.15]{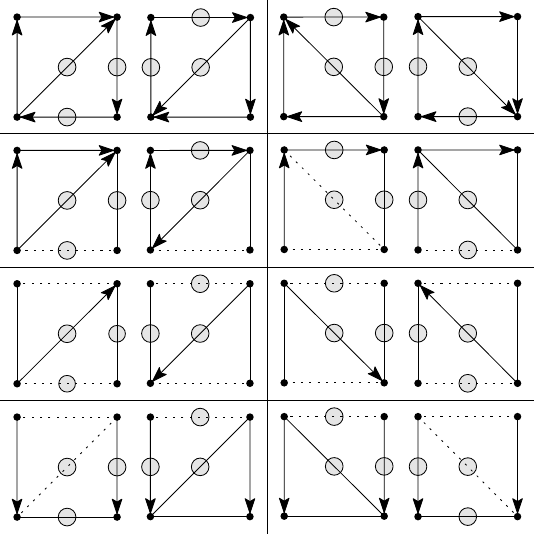}
\caption{
Reversing when $\pi(\cD)$ is a square
in types $B_n$, $C_n$ or $D_n$. 
Each row corresponds to one of the four possible configurations 
of $\cD$. 
In each cell, the shaded circles indicate 
how to reverse $\cD$
in two steps, using one of the games between an ``antipodal'' pair of players
along $\cD$. Note that each row has exactly two cells, as there is always
exactly two ways to reverse $\cD$ in two steps. 
The ``middle'' game (not in the square itself)
is reversed twice, so returned to its original orientation. 
}
\label{F_rev_square}
\end{figure}

{\bf Case 1b.}
Suppose that 
$\Phi=B_n$
or $D_n$ and 
$\cG_1,\cG_2$ are adjacent neutral triangles on three vertices,
so that $\pi(\cD)$ is a tent. In this case, both of the (competitive
and collaborative) games between the ``base'' vertices lead
to a way of reversing $\cD$. 
See 
Figure \ref{F_rev_tent1}. 

\begin{figure}[h!]
\includegraphics[scale=1.15]{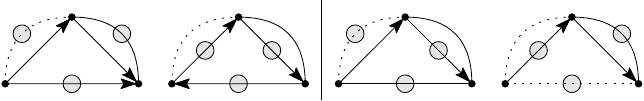}
\caption{
Reversing when $\pi(\cD)$ is a tent formed by two neutral triangles
in types $B_n$ and $D_n$. 
This figure has only
one row, as there is only one possibility 
for $\cD$. 
}
\label{F_rev_tent1}
\end{figure}

{\bf Case 1c.}
Suppose that 
$\Phi=B_n$ and that one of 
$\cG_1,\cG_2$ is a neutral pair and the other is an adjacent
neutral pair or triangle, 
so that $\pi(\cD)$ is a fork. Then 
the half edge game played by the ``middle'' vertex 
and
exactly one of the 
games between the ``base'' vertices 
lead to ways of reversing $\cD$.
See 
Figure \ref{F_rev_fork}. 

\begin{figure}[h!]
\includegraphics[scale=1.15]{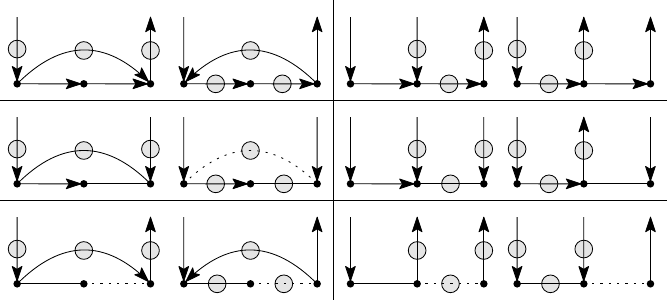}
\caption{
Reversing when $\pi(\cD)$ is a fork
in type $B_n$. 
Each row corresponds to a possible configuration 
for $\cD$. 
}
\label{F_rev_fork}
\end{figure}

By Cases 1a--c, statement (1) follows, that is,  in types
$B_n$ and $D_n$ the network $N$ is
always a single diamond, as in Figure \ref{F_adj_BD}. 

\begin{figure}[h!]
\includegraphics[scale=1.15]{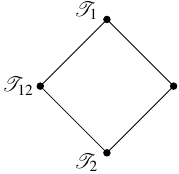}
\caption{
When $\cG_1,\cG_2$ are adjacent,
the only possible $N(\cT_1,\cT_2)$ in types $B_n$ and $D_n$
is a diamond. 
}
\label{F_adj_BD}
\end{figure}

{\bf Case 2.} In type $C_n$, the case that $\pi(\cD)$ is a 
square follows by the same argument as in types $B_n$
and $D_n$. Indeed, recall that any reversal of $\cD$ in two 
steps will involve reversing exactly two games in $\cD$ in 
each step. Therefore, no clovers will be involved in 
such a reversal of $\cD$, and so once again 
$N$ is a single diamond,
yielding statement (2). 

{\bf Case 3a.} Suppose that $\Phi=C_n$
and that $\pi(\cD)$ is a tent formed by 
two adjacent neutral triangles on three vertices. Then by 
Case 1b, $N$ contains a single diamond. 
However, using the loop game $\ell$ played by the ``middle'' vertex, 
we obtain an additional path of length two between $\cT_1,\cT_2$. 
We can use $\ell$ to reverse the two games
on one ``side'' of the tent. Then, in turn, we can use $\ell^*$
to reverse the other two games. See Figure \ref{F_rev_tent2}. 
Hence $N$ is a split diamond in this case. 

\begin{figure}[h!]
\includegraphics[scale=1.15]{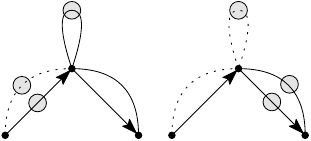}
\caption{
The additional way of reversing when $\pi(\cD)$ is a tent
in type $C_n$. 
}
\label{F_rev_tent2}
\end{figure}

{\bf Case 3b.} Suppose that $\Phi=C_n$
and that $\pi(\cD)$ is a tent formed by 
two adjacent neutral clovers. Then $\cD$ is on three vertices, 
and so by Case 1b, we find that 
$N$ is a split diamond, once again. 

By Cases 3a--b, statement (3) follows. 
The difference between Cases 3a and 3b are depicted in 
the first column of Figure \ref{F_adj_C}.

{\bf Case 4.} Finally, suppose that 
$\Phi=C_n$
and that $\pi(\cD)$ is a hanger. We will argue that 
this case corresponds to the second and third columns in 
Figure \ref{F_adj_C}. 

\begin{figure}[h!]
\includegraphics[scale=1.15]{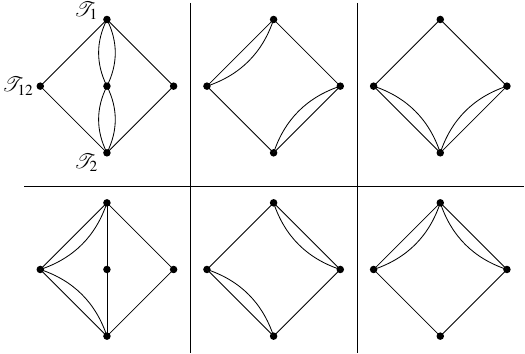}
\caption{
The additional (other than the single diamond)
possible $N(\cT_1,\cT_2)$ 
in type $C_n$, when $\cG_1,\cG_2$ are adjacent. 
{\it From left to right:} Split, double and heavy diamonds. 
In each cell, the leftmost vertices are $\cT_1$, $\cT_{12}$ and $\cT_2$, 
as in the first cell. 
}
\label{F_adj_C}
\end{figure}

Note that, in this case, exactly one of $\cG_1,\cG_2$ is a neutral clover
and the other is an adjacent neutral triangle. Suppose that 
the loop game $\ell$ in $\cD$ is played by vertex $x$ and that the other 
two vertices in $\cD$ are $y,z$. Note that any way of reversing 
$\cD$ in two steps will involve reversing $\ell$ exactly once, 
and so each path of length two from $\cT_1$ to $\cT_2$ 
will contain
exactly one double edge. 

The four cases
in the second and third columns in 
Figure \ref{F_adj_C} can be seen by considering the other 
games
played between $x,y$ and $x,z$ that are not in $\cD$. 
Depending on their outcomes, each such game 
either creates a clover with loop $\ell$ at $x$ 
or else forms a neutral triangle together with 
the two ``opposite'' games
in $\cD$. After this clover/triangle is reversed, 
the triangle/clover, which was not initially, becomes 
present. 
See Figure \ref{F_rev_hanger}.

\begin{figure}[h!]
\includegraphics[scale=1]{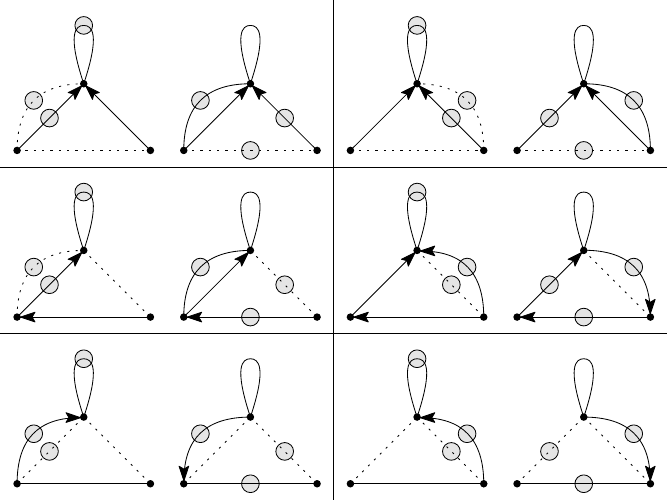}
\caption{
Reversing when $\pi(\cD)$ is a hanger
in type $C_n$. 
Each row corresponds to a possible configuration 
for $\cD$. 
}
\label{F_rev_hanger}
\end{figure}

The proof is complete. 
\end{proof}

%%%%%%%%%%%%%%%%%%%%%%%%%%%%%%%%%%%%%%
%%%%%%%%%%%%%%%%%%%%%%%%%%%%%%%%%%%%%%
%%%%%%%%%%%%%%%%%%%%%%%%%%%%%%%%%%%%%%
%%%%%%%%%%%%%%%%%%%%%%%%%%%%%%%%%%%%%%
%%%%%%%%%%%%%%%%%%%%%%%%%%%%%%%%%%%%%%
%%%%%%%%%%%%%%%%%%%%%%%%%%%%%%%%%%%%%%
\subsection{Extended networks}
\label{S_ex_net}

Lemmas \ref{L_NclassDis} and \ref{L_NclassAdj} above 
classify the types of interchange networks
$N(\cT_1,\cT_2)$. 
Recall that such a network contains all paths of length
two between $\cT_1,\cT_2$. 

\begin{definition}
We define the {\it extended interchange network}
$\Nex(\cT_1,\cT_2)$ to be the union of 
$N(\cT_1',\cT_2')$ over all 
``antipodal'' pairs $\cT_1',\cT_2'$ in 
$N(\cT_1,\cT_2)$ at distance two. 
\end{definition}

Single, double and quadruple diamonds 
are ``stable,'' in the sense
that $\Nex=N$.
In contrast, split and heavy diamond networks 
extend to a type of structure, which we call a 
\textit{crystal}. 
See Figure \ref{F_crystal}.

\begin{figure}[h!]
\includegraphics[scale=1.25]{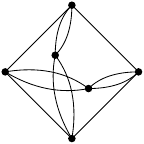}
\caption{
The 
crystal extended interchange network.}
\label{F_crystal}
\end{figure}

\begin{remark}
All of the interchange networks that we have described, 
except the single diamond, 
can be found in Figure \ref{F_drum} above. 
This demonstrates how interchange networks can 
overlap and 
mesh together to form the interchange graph of a 
given score sequence.
\end{remark}

\begin{lemma}
\label{L_Nstar}
Suppose that $\cT_1,\cT_2$  are at distance two 
in $\ig(\Phi,{\bf s})$. Let $N=N(\cT_1,\cT_2)$ 
and $\hat N=\hat N(\cT_1,\cT_2)$. 
\begin{enumerate}
\item If $N$ is a single, double or quadruple
diamond, then the extended interchange network $\Nex=N$.
\item Otherwise, if $N$ is a split or heavy 
diamond, then 
the extended interchange network 
$\Nex$
is a crystal. 
\end{enumerate}
\end{lemma}

\begin{proof}
Statement (1) is clear, and can be seen by inspection. 
On the other hand, statement (2) follows by 
repeated application of Lemma \ref{L_NclassAdj}, 
considering the various
antipodal pairs in $N$. 

{\bf Case 1.}  If $N$ is a split diamond, as in the 
first column of 
Figure \ref{F_adj_C}, then consider $\cT_1',\cT_2'$ in 
$N$ that are incident to only single edges in 
$N$. By Lemma \ref{L_NclassAdj}, it follows that 
$N(\cT_1',\cT_2')$ is a split diamond, and therefore
$\Nex$ is a crystal.

{\bf Case 2.}  If $N$ is a heavy diamond, as in the 
third column of 
Figure \ref{F_adj_C}, then consider $\cT_1',\cT_2'$ in 
$N$, each of which incident to exactly one single
edge and one double edge in 
$N$. By Lemma \ref{L_NclassAdj}, it follows that 
$N(\cT_1',\cT_2')$ is a split diamond, and therefore
there is a path of length two between them 
consisting of two single edges in 
$\Nex$. Then, applying 
Lemma \ref{L_NclassAdj}, once again, but this time to the midpoint 
$\cT_{12}'$
along this path
and $\cT_3$ in $N$ that is incident to two single
edges in $N$, we find that 
$\Nex$ is a crystal, as claimed. 
\end{proof}

Recall that two generators are either 
disjoint or have exactly
one game in common. A similar
property holds for extended networks. 

\begin{lemma}
\label{L_NstarED}
Any two distinct extended networks 
$\Nex\neq \Nex'$ are either edge-disjoint 
or have exactly one single or double edge in 
common. 
\end{lemma}

\begin{proof}
By Lemma \ref{L_Nstar} there are only four types of extended networks. 
By an elementary case analysis, it can be seen that  any two 
antipodal (distance two)  vertices
$\cT_1,\cT_2$ in an extended network $\Nex$ give
rise to the same extended network $\Nex$. That is, 
$\Nex=\Nex(\cT_1,\cT_2)$, for any such $\cT_1,\cT_2$. 
From this observation the result follows,
since if two extended networks $\Nex, \Nex'$ 
share at least three vertices, 
then they necessarily 
have at least one antipodal 
pair of vertices in common. 
\end{proof}

%%%%%%%%%%%%%%%%%%%%%%%%%%%%%%%%%%%%%%
%%%%%%%%%%%%%%%%%%%%%%%%%%%%%%%%%%%%%%
%%%%%%%%%%%%%%%%%%%%%%%%%%%%%%%%%%%%%%
%%%%%%%%%%%%%%%%%%%%%%%%%%%%%%%%%%%%%%
%%%%%%%%%%%%%%%%%%%%%%%%%%%%%%%%%%%%%%
%%%%%%%%%%%%%%%%%%%%%%%%%%%%%%%%%%%%%%
\subsection{Properties of crystals}

In this section, we obtain two key properties
of crystals, which will play a crucial role in the   
type $C_n$ 
couplings
discussed in Section \ref{S_MainC} below. 

First, we note that 
crystals cannot share a single edge. 
We will use this, together with Lemma \ref{L_NstarED}, 
to extend natural  
couplings on networks to a coupling
on the full interchange graph. 

\begin{lemma}
\label{L_crystals2}
Suppose that $\Nex \neq \Nex'$ are distinct crystals
in an interchange network $\ig(C_n,{\bf s})$.
Then $\Nex$, $\Nex'$
are either edge-disjoint or share a double edge. 
That is, no such $\Nex\neq \Nex'$ share a single edge. 
\end{lemma}

\begin{proof}
By the proof of Lemmas \ref{L_NclassAdj} and \ref{L_Nstar}, it can be
seen that each crystal is associated with three players. 
Each double edge in the crystal corresponds to 
reversing a neutral clover involving two of them, 
and each single edge 
corresponds to reversing
a neutral triangle involving all three. 

By Lemma \ref{L_NstarED}, it suffices to show that 
two crystals $\Nex \neq \Nex'$
cannot share a single edge. 
To see this, simply note that otherwise 
both tournaments joined by this single edge
would contain a tournament on three players
with three neutral triangles, which is impossible. 
See Figure \ref{F_crystals2}. 
\end{proof}

\begin{figure}[h!]
\includegraphics[scale=1.25]{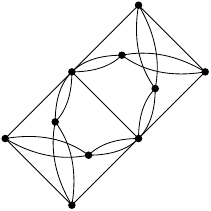}
\caption{
It is impossible for two (distinct) crystals
to share a single edge,
as depicted above,
since there are at most two neutral triangles on 
any given three players
in a tournament.
However, each of the two ``middle'' vertices
in this figure are incident to three single edges. 
}
\label{F_crystals2}
\end{figure}

The previous result shows that single edges can be in at most
one crystal. Double edges, on the other hand, can be in more
than one. The following result gives an upper bound on this number, 
which is related to the re-weighting $w$ of the
graph metric, discussed in Section \ref{S_MainC}, 
under which our coupling will be contractive. 

Recall that $d=d(C_n,{\bf s})$
is the degree of $\ig(C_n,{\bf s})$. 

\begin{lemma}
\label{L_crystals_ee}
Any given double edge 
in an interchange network $\ig(C_n,{\bf s})$
is contained in at most
$\min\{d,2n\}$ crystals. 
\end{lemma}

\begin{proof}
Consider a double edge between some $\cT_1,\cT_2$.
By Lemma \ref{L_NstarED}, each crystal containing it 
corresponds to an additional double edge
or two single edges incident to $\cT_1$. 
It follows that there are at most
$(d-2)/2$ such crystals. 

The second bound is somewhat more complicated. 
Recall, as noted in the proof of Lemma \ref{L_crystals2},  
that each crystal is associated with three
players. 
Suppose that a double edge between some $\cT_1,\cT_2$
is associated with a neutral clover involving players $i,j$. 
We claim that, for any other player $k$, 
there are at most two crystals
associated with $i,j,k$. 
To see this, observe that, if there were three, 
then one of $\cT_1,\cT_2$ would be incident to four single edges 
in these crystals. However, this would imply that in one of 
$\cT_1,\cT_2$ there are four neutral triangles
on $i,j,k$, which is impossible. Indeed, as noted in the proof of 
Lemma \ref{L_crystals2}, there can be at most two. 
Therefore, there are at most $2(n-2)$ crystals containing 
any given double edge. 
(In fact, the upper bound $n-2$ can be proved, but involves 
a more careful analysis.)
\end{proof}

%%%%%%%%%%%%%%%%%%%%%%%%%%%%%%%%%%%%%%
%%%%%%%%%%%%%%%%%%%%%%%%%%%%%%%%%%%%%%
%%%%%%%%%%%%%%%%%%%%%%%%%%%%%%%%%%%%%%
%%%%%%%%%%%%%%%%%%%%%%%%%%%%%%%%%%%%%%
%%%%%%%%%%%%%%%%%%%%%%%%%%%%%%%%%%%%%%
%%%%%%%%%%%%%%%%%%%%%%%%%%%%%%%%%%%%%%
\section{Rapid mixing}
\label{S_Coupling}

Using the results of the previous section, 
we show that simple random walk 
on any given $\ig(\Phi,{\bf s})$ is rapidly mixing. 
The idea is to first define couplings
on extended networks $\Nex$. 
We then argue that these 
couplings are
compatible, and extend to a full coupling.

In types $B_n$ and $D_n$, rapid mixing then follows by
Theorem \ref{T_PC}, using the standard weighting $w=\delta$ given 
by the graph distance $\delta$ in $\ig(\Phi,{\bf s})$. 
In type $C_n$, we will need to select a special re-weighting 
$w\neq\delta$, accounting for the presence of crystals in the interchange 
graphs of this type. 

For a tournament $\cT\in \tour(\Phi,{\bf s})$, we let $\cE(\cT)$
denote the set of edges in $\ig(\Phi,{\bf s})$ incident to $v(\cT)$.

%%%%%%%%%%%%%%%%%%%%%%%%%%%%%%%%%%%%%%
%%%%%%%%%%%%%%%%%%%%%%%%%%%%%%%%%%%%%%
%%%%%%%%%%%%%%%%%%%%%%%%%%%%%%%%%%%%%%
%%%%%%%%%%%%%%%%%%%%%%%%%%%%%%%%%%%%%%
%%%%%%%%%%%%%%%%%%%%%%%%%%%%%%%%%%%%%%
%%%%%%%%%%%%%%%%%%%%%%%%%%%%%%%%%%%%%%
\subsection{Coupling in $B_n$ and $D_n$}
\label{S_MainBD}

We begin with the simplest cases of types 
$B_n$ and $D_n$. These types are the most straightforward, 
since then all networks $N(\cT_1,\cT_2)$ are single diamonds,
and no re-weighting of the graph metric is necessary. 

\begin{theorem}
\label{T_MainBD}
Let $\Phi=B_n$ or $D_n$. Fix any ${\bf s}\in \score(\Phi)$. 
Then lazy simple random walk $(\cT_n:n\ge0)$
on the Coxeter interchange graph $\ig(\Phi,{\bf s})$ is rapidly mixing 
in time $t_{\rm mix}= O(d \log n)$. 
\end{theorem}

\begin{proof}
Let $\Phi=B_n$ or $D_n$. 
Consider two copies of lazy simple random walk $(\cT_n')$ and 
$(\cT_n'')$ on $\ig(\Phi,{\bf s})$, 
started from neighboring $\cT_0',\cT_0''\in \tour(\Phi,{\bf s})$. 
Then $\cT_0''=\cT_0'*\cG$ 
for 
some type $\Phi$ 
generator $\cG\subset\cT_0'$. 
In this sense, the random walks start at distance 1. 

We will construct a contractive coupling of 
$\cT_1',\cT_1''$, 
such that the expected distance between 
$\cT_1',\cT_1''$ is strictly less than $1$, 
for every 
choice of $\cT_0',\cT_0''$. 
In fact, in this coupling, the $\cT_1',\cT_1''$ will coincide
with probability $1/d$, and otherwise  
remain at distance $1$. 

More specifically, we couple $\cT_1',\cT_1''$ using  
the natural bijection $\psi$ from $\cE(\cT_0')$
to $\cE(\cT_0'')$, which fixes the edge $\{\cT_0',\cT_0''\}$
and 
pairs ``opposite''
edges in each single diamond containing $\{\cT_0',\cT_0''\}$. 
By Lemmas \ref{L_NclassDis} and \ref{L_NclassAdj},
for each type $\Phi$ generator $\cG'\neq\cG\subset\cT_0'$, 
the network
$N(\cT_0'',\cT_0'*\cG')$ is a 
single diamond. 
As such, there are exactly two paths of length two from 
$\cT_0''$ to $\cT_0'*\cG'$. 
One such path passes through $\cT_0'$. 
We let  
\begin{equation}\label{E_phi}
\psi(\{\cT_0',\cT_0'*\cG'\})=\{\cT_0'',\cT_0''*\cG''\}
\end{equation}
be the first edge along the other such path. 
By Theorem \ref{T_deg} and Lemmas \ref{L_Nstar} and \ref{L_NstarED},
$\psi$ is a bijection 
from $\cE(\cT_0')$
to $\cE(\cT_0'')$.

Finally, we define the coupling of $\cT_1',\cT_1''$ as follows. 
Let $\{\cT_0',\cT_0'* \cG'\}$ be a 
uniformly random edge in  $\cE(\cT_0')$
and $r_0$ a Bernoulli$(1/2)$ 
random variable (i.e., a fair ``coin flip'').
If $\cG'=\cG$, we put 
$\cT_1'=\cT_1''=\cT_0'$ if $r_0=0$
and $\cT_1'=\cT_1''=\cT_0''$ if $r_0=1$. 
In this case, the coupling contracts. 
On the other hand, if $\cG'\neq \cG$
we put $\cT_1'=\cT_0'$ 
and $\cT_1''=\cT_0''$
if $r_0=0$,
and $\cT_1'=\cT_0'*\cG'$
and $\cT_1''=\cT_0''*\cG''$
if $r_0=1$, 
where $\cG''$ is given by the bijection 
$\psi$ in \eqref{E_phi}.  See Figure \ref{F_coup_BD}.

\begin{figure}[h!]
\includegraphics[scale=1.15]{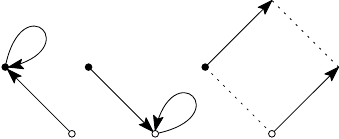}
\caption{A contractive coupling in 
$B_n$ and $D_n$. The black and white vertices
represent 
the starting positions 
$\cT_0',\cT_0''$. 
In the configuration at right, 
corresponding to $\cG'\neq\cG$, 
the loops have been omitted, 
since in this case the walks are either
both lazy or not.
}
\label{F_coup_BD}
\end{figure}

In this coupling, $\cT_1'=\cT_1''$ with probability $1/d$.  
Otherwise, they remain at distance $1$. 
By Theorem \ref{T_MainInt}, 
the diameter of $\ig(\Phi,{\bf s})$ is $D= O(n^2)$.  
Therefore, 
by Theorem \ref{T_PC},  the mixing
time is bounded by $O(d\log n)$. 
\end{proof}

\begin{remark}
Rapid mixing for classical (type $A_{n-1}$) tournaments 
follows as a special case of the argument above.  
\end{remark}

%%%%%%%%%%%%%%%%%%%%%%%%%%%%%%%%%%%%%%
%%%%%%%%%%%%%%%%%%%%%%%%%%%%%%%%%%%%%%
%%%%%%%%%%%%%%%%%%%%%%%%%%%%%%%%%%%%%%
%%%%%%%%%%%%%%%%%%%%%%%%%%%%%%%%%%%%%%
%%%%%%%%%%%%%%%%%%%%%%%%%%%%%%%%%%%%%%
%%%%%%%%%%%%%%%%%%%%%%%%%%%%%%%%%%%%%%
\subsection{Coupling in $C_n$}
\label{S_MainC}

Finally, we investigate mixing in type $C_n$.  

Recall that,  in types $B_n$ and $D_n$,
the coupling was determined by an edge pairing,
given by a bijection $\psi$
from $\cE(\cT_0')$
to $\cE(\cT_0'')$, where $\cT_0',\cT_0''$ are neighboring
tournaments in $\ig(\Phi,{\bf s})$.
The coupling in type $C_n$ is also determined 
by such a $\psi$, however, since there
are a number of different networks in type $C_n$, 
the pairing is more involved. 

The fact (see Lemma \ref{L_crystals2}) 
that distinct crystals cannot
share a single edge is crucial. Otherwise
it would not be possible to extend 
couplings on extended networks
to a full coupling. Roughly speaking, this is because
(see Case 1b in the proof of Theorem \ref{T_MainC} below) 
single edges $\{\cT_0',\cT_0''\}$ in a crystal will need to be paired
with one of the edges in a double edge of the same crystal. 
As such, if there were two crystals with the same single edge, 
a bijective pairing would not be possible.

Furthermore, there is an additional complication in type $C_n$. 
As it turns out, the pairing $\psi$ does not lead to a contractive coupling, 
with respect to the graph distance in $\ig(C_n,{\bf s})$. The problem
concerns the case that the initial starting positions $\cT_0',\cT_0''$
are joined by a single edge in a crystal. 
In this case, the natural coupling is only ``neutral'' 
(i.e., with $\alpha=0$ in Theorem \ref{T_PC}), 
rather than contractive.  

There are (at least) three ways to overcome this difficulty, 
leading to increasingly better bounds on the mixing time. 
The first way is to apply 
Bordewich and Dyer's \cite{BD07}  
path coupling without contraction, leading to 
an upper bound $O(d n^4)=O(n^7)$. It is also possible 
to apply path coupling at time $t=2$, since at this point the coupling
(with respect to the usual graph metric) becomes contractive. 
In doing so, the key is to 
observe that, in the problematic case that 
$\cT_0',\cT_0''$
are joined by a single edge in a crystal, 
if $\cT_1'$ stays within the crystal then 
the pairing $\psi$ (described below) selects a  $\cT_1''$ 
in the crystal
such that 
$\cT_1',\cT_1''$ are now joined by a double edge. 
This argument leads to an upper bound of 
$O(d^2 \log n)=O(n^6\log n)$. 

We will present a third strategy, by re-weighting 
the metric, which yields a better bound. 

Recall (see Lemma \ref{L_crystals2}) 
that single edges in $\ig(C_n,{\bf s})$ 
are contained in at most one crystal. 
Double edges, on the other
hand, can be contained in more than one. 
We define the {\it crystal degree} of a double edge to be 
the number of crystals
containing it. 
We let $\gamma=\gamma(C_n,{\bf s})$ denote the 
{\it maximal crystal degree,} 
over all double edges. 
By Lemma \ref{L_crystals_ee}, we have that 
$\gamma\le \min\{d,2n\}$, 
where $d=d(C_n,{\bf s})$ is the degree of 
$\ig(C_n,{\bf s})$. 

We will assume throughout that $\gamma>0$. 
Indeed, if $\gamma=0$, then there
are no crystals in the interchange graph. In this case, 
a straightforward modification of 
the proof of Theorem \ref{T_MainBD} 
(using the standard graph metric) 
shows that $t_{\rm mix}= O(d \log n)$.

We will prove the following 
result, 
using the weighting $w$ that puts $w=1$ 
on each edge in a double edge
and $w=1+1/\gamma$ on each single edge. 

To be clear, 
this choice of $w$ re-weights the graph distance 
between neighboring vertices
joined by single edges, but not those joined by double edges. 
Specifically, 
if $u,v$ are joined by a single edge then 
$w(u,v)=1+1/\gamma$, and if $u,v$ 
are joined by a double edge
then each edge is given weight $1$, 
and so the weighted distance between $u,v$ remains 
$w(u,v)=1$ (see Definition \ref{D_w}).

\begin{theorem}
\label{T_MainC}
Let $\Phi=C_n$. Fix any ${\bf s}\in \score(\Phi)$. 
Then lazy simple random walk $(\cT_n:n\ge0)$
on the Coxeter interchange graph $\ig(C_n,{\bf s})$ is rapidly mixing. 
If there are no crystals in $\ig(C_n,{\bf s})$ (when $\gamma=0$) then 
 $t_{\rm mix}= O(d \log n)$. Otherwise, we have 
 $t_{\rm mix}= O(\gamma d \log n)$. 
\end{theorem}

In particular, this result implies that $t_{\rm mix}=O(n^4\log n)$.

\begin{proof}
As discussed, let us assume that $\gamma>0$, as otherwise
a simple adaptation of the general reasoning in types $B_n$ and $D_n$ 
(the proof of Theorem \ref{T_MainBD})
shows that $t_{\rm mix}= O(d \log n)$.

Consider two copies of 
lazy simple random walk $(\cT_n')$ and 
$(\cT_n'')$ on
$\ig(C_n,{\bf s})$, 
started from 
neighbors 
$\cT_0',\cT_0''\in \tour(C_n,{\bf s})$. Let $\cG$ be such that 
$\cT_0''=\cT_0'*\cG$. 

The first step is to obtain an edge pairing $\psi$, which will 
associate a $\cT_1''$ to each possible $\cT_1'$. 
Then we will show that this coupling is contractive, 
with respect to the re-weighting  $w=1+1/\gamma$ on single edges
and $w=1$ on each edge in double edges.
The key in this regard will be the classification of extended 
networks, established in Section \ref{S_Networks}.

By Lemma \ref{L_crystals2}, there are three cases to consider: 
\begin{enumerate}
\item$\cG=\Delta$ is a neutral triangle, 
and the edge $\{\cT_0',\cT_0''\}$ is in 
\begin{enumerate}
\item no crystal, 
\item  exactly one crystal.
\end{enumerate}
\item $\cG=\Theta$ is a neutral clover, 
and the double edge between $\cT_0',\cT_0''$ is 
in $\gamma'\le\gamma$ crystals. 
\end{enumerate}

In these cases, 
we will construct couplings 
with the following properties:   
\begin{itemize}
\item In Case 1a, either 
$\cT_1'=\cT_1''$, or else $\cT_1',\cT_1''$ are again joined by a single edge. 
\item In Case 1b, 
either $\cT_1',\cT_1''$ are joined by a double 
edge in the crystal, or else $\cT_1',\cT_1''$ are joined by a single
edge.  
\item In Case 2,  
either $\cT_1',\cT_1''$ are joined by a single 
edge in some crystal containing the double edge between $\cT_0',\cT_0''$, 
or else $\cT_1',\cT_1''$ are joined by a double
edge.  
\end{itemize}

Note that, under these couplings, crystal networks work like 
``switches,'' in that they move single edges to double edges, and 
vice versa. 
Also note that, it is Cases 1b and 2 in which the choice of $w$
is crucial. We put weight $w=1+1/\gamma$ on single edges
so that, as we will see, the couplings in these cases are contractive.

{\bf Case 1a.} Suppose that $\cG=\Delta$
is a neutral triangle, and that the single
edge $\{\cT_0',\cT_0''\}$
is not contained in a crystal. 
Then,  
by Lemmas \ref{L_NclassDis} and \ref{L_NclassAdj}, 
all extended networks $\Nex$ 
containing $\{\cT_0',\cT_0''\}$ are 
single and double diamonds. 
As such, it is only slightly more complicated to construct
a contractive coupling in this case,  
than it was in types $B_n$ and $D_n$ above. 
We proceed as depicted in Figure \ref{F_coup_C1a}
(cf.\ Figure \ref{F_coup_BD}). 

\begin{figure}[h!]
\includegraphics[scale=1.15]{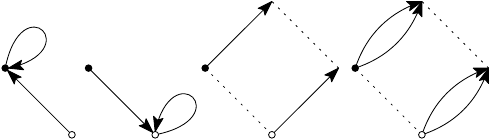}
\caption{{\it Case 1a:} A contractive coupling, 
when the starting positions (black and white vertices) 
are joined by a single edge, which is not 
in a crystal. 
}
\label{F_coup_C1a}
\end{figure}

Once again (as in the proof of Theorem \ref{T_MainBD}), 
using Theorem \ref{T_deg} and 
Lemmas \ref{L_Nstar} and \ref{L_NstarED},
we find a bijection $\psi$ from 
from $\cE(\cT_0')$
to $\cE(\cT_0'')$ that fixes the edge $\{\cT_0',\cT_0''\}$
and pairs ``opposite'' edges in single and double diamonds
containing $\{\cT_0',\cT_0''\}$. 
and so correspond to the same generator.

For each edge $\{\cT_0',\cT_0'*\cG'\}$
in a single diamond $\Nex$
containing $\{\cT_0',\cT_0''\}$, 
we let 
\begin{equation}\label{E_phiC1}
\psi(\{\cT_0',\cT_0'*\cG'\})=\{\cT_0'',\cT_0''*\cG''\}
\end{equation}
be the ``opposite'' edge in $\Nex$. 

Likewise, 
for each double edge, consisting of two copies 
$\{\cT_0',\cT_0'*\cG'\}^{(i)}$, 
with $i\in\{1,2\}$, of the same edge 
in a double diamond network $\Nex$
containing $\{\cT_0',\cT_0''\}$,  
we let 
\begin{equation}\label{E_phiC2}
\psi(\{\cT_0',\cT_0'*\cG'\}^{(i)})=\{\cT_0'',\cT_0''*\cG''\}^{(i)},
\end{equation}
with $i\in\{1,2\}$, 
be the ``opposite'' edges in $\Nex$. 

We couple $\cT_1',\cT_1''$ as follows. 
Let $\{\cT_0',\cT_0'* \cG'\}$ be a 
uniformly random edge in  $\cE(\cT_0')$ 
and $r_0$ a Bernoulli$(1/2)$.
(Note that, this is a uniformly random edge, not generator.
Indeed, neutral clovers $\cG'=\Theta$ 
corresponding to double edges are twice as likely
to be selected as neutral triangles $\cG'=\Delta$.) 
If $\cG'=\cG$, we put 
$\cT_1'=\cT_1''=\cT_0'$ if $r_0=0$
and $\cT_1'=\cT_1''=\cT_0''$ if $r_0=1$. 
On the other hand, if $\cG'\neq \cG$, 
we put $\cT_1'=\cT_0'$ 
and $\cT_1''=\cT_0''$
if $r_0=0$
and $\cT_1'=\cT_0'*\cG'$
and $\cT_1''=\cT_0''*\cG''$
if $r_0=1$, where $\cG''$ is given by the bijection 
$\psi$, defined  in \eqref{E_phiC1} or \eqref{E_phiC2} above. 

Note that $\cT_1'=\cT_1''$ with 
probability $1/d$. 
Otherwise, $\cT_1',\cT_1''$ 
are again joined by a single edge. As such
\begin{equation}\label{E_E1a}
\E[w(\cT_1',\cT_1'')]= (1-1/d)w(\cT_0',\cT_0''). 
\end{equation}

{\bf Case 1b.} Suppose that $\cG=\Delta$
is a neutral triangle, and that the single
edge $\{\cT_0',\cT_0''\}$
is contained in {\it exactly} one crystal. 
Once again, by Lemma \ref{L_NstarED}, 
all single and double diamonds and the one crystal
containing $\{\cT_0',\cT_0''\}$ are otherwise 
edge-disjoint. 
In this case, the bijection can no longer 
fix $\{\cT_0',\cT_0''\}$, as in Case 1a. 
Rather, we will need to use this edge in a non-trivial way
in order to 
define the edge pairing within the crystal.

The bijection $\psi$, in this case, is defined in the 
same way as in Case 1a for the edges in each single
and double diamond. On the other hand, for the edges in the crystal, 
we define $\psi$ as indicated in Figure \ref{F_coup_C1b}. 
That is, the two single edges in the crystal incident to 
$\cT_0'$ 
(one of which is $\{\cT_0',\cT_0''\}$) 
are paired with the double edges in the crystal 
incident to $\cT_0''$, and vice versa. 
By Theorem \ref{T_deg}, $\psi$ is a bijection 
from $\cE(\cT_0')$
to $\cE(\cT_0'')$.
 We stress here that 
$\psi$ pairs edges, not generators, 
and it is critical, in this case, that there is only
one crystal containing $\{\cT_0',\cT_0''\}$ (since it can 
only be paired once). 

\begin{figure}[h!]
\includegraphics[scale=1.15]{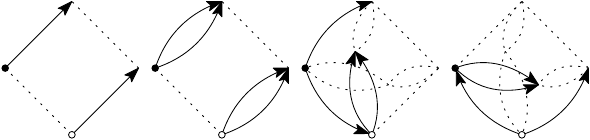}
\caption{{\it Case 1b:} A neutral 
coupling, 
when the starting positions (black and white vertices) 
are joined by a single edge, which is in exactly one crystal. 
Note that, if the walks stay in the crystal, they
move to a pair joined by a double edge. 
}
\label{F_coup_C1b}
\end{figure}

To couple $\cT_1',\cT_1''$, 
we let 
$\{\cT_0',\cT_0'* \cG'\}$ be a 
uniformly random edge in $\cE(\cT_0')$ 
and $r_0$ a Bernoulli$(1/2)$.
If $r_0=0$, we put 
$\cT_1'= \cT_0'$ and $\cT_1''=\cT_0''$.
If $r_0=1$, 
we put 
$\cT_1'=\cT_0'*\cG'$
and $\cT_1''=\cT_0''*\cG''$, 
where $\cG''$ is given by the bijection $\psi$. 

In this case, with 
probability $2/d$ 
the pair  $\cT_1',\cT_1''$ 
remains in the crystal, but
are now joined by a double edge. 
Otherwise, $\cT_1',\cT_1''$ 
are again joined by a single edge.
Therefore, 
\begin{align}
\E[w(\cT_1',\cT_1'')]
&= (1-2/d)(1+1/\gamma)+2/d\nonumber\\
&=\left[1-\frac{2}{d(1+\gamma)}\right]w(\cT_0',\cT_0''),\label{E_E1b}
\end{align}
since $w(\cT_0',\cT_0'')=1+1/\gamma$.

{\bf Case 2.} Finally, suppose that $\cG=\Theta$
is a neutral clover. Suppose that $\{\cT_0',\cT_0''\}$
is contained in $\gamma'\le \gamma$ crystals. 

By Lemmas \ref{L_NclassDis} and \ref{L_NclassAdj}, 
all extended networks $\Nex$ 
containing $\{\cT_0',\cT_0''\}$ are 
double and quadruple diamonds
and crystals. As in the previous cases, 
we define $\psi$ in this case by 
pairing ``opposite'' edges in the 
double and quadruple diamonds. 
In this case, $\psi$ fixes the two edges in the double edge
between $\cT_0',\cT_0''$. 
Note that if $\{\cT_0',\cT_0''\}$ is in a crystal, 
then one of $\cT_0',\cT_0''$ is incident 
to two single edges
in the crystal and the other is incident to a double
edge $\neq \{\cT_0',\cT_0''\}$ in the crystal. 
We define $\psi$ on each such crystal by
pairing these edges, as indicated in 
Figure \ref{F_coup_C2}. 
Once again, applying Theorem \ref{T_deg}, 
we see that $\psi$ is a bijection 
from $\cE(\cT_0')$
to $\cE(\cT_0'')$.

\begin{figure}[h!]
\includegraphics[scale=1.15]{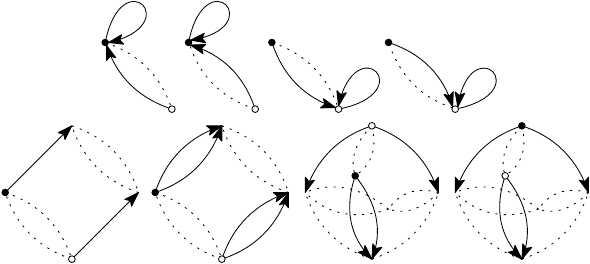}
\caption{{\it Case 2:} A contractive coupling, 
when the starting positions (black and white vertices) 
are joined by a double edge. 
}
\label{F_coup_C2}
\end{figure}

To couple $\cT_1',\cT_1''$, 
we let $\{\cT_0',\cT_0'*\cG'\}$
be a uniformly 
random edge in $\cE(\cT_0')$
and $r_0$ a Bernoulli$(1/2)$.
If $\cG'=\cG$, we put 
$\cT_1'=\cT_1''=\cT_0'$ if $r_0=0$
and $\cT_1'=\cT_1''=\cT_0''$ if $r_0=1$. 
Otherwise, if $\cG'\neq \cG$, 
we put $\cT_1'=\cT_0'$ 
and $\cT_1''=\cT_0''$
if $r_0=0$
and $\cT_1'=\cT_0'*\cG'$
and $\cT_1''=\cT_0''*\cG''$
if $r_0=1$, where $\cG''$ is given by the bijection $\psi$. 

In this case, $\cT_1'=\cT_1''$  with 
probability $2/d$. With probability $\gamma'/d$, the pair 
$\cT_1',\cT_1''$ move within one of the 
$\gamma'$ crystals containing $\{\cT_0',\cT_0''\}$, 
and are then  joined by a single edge. 
Otherwise, with probability $1-(2+\gamma')/d$, 
$\cT_1',\cT_1''$ remain joined by
a double edge. Therefore, 
\begin{align}
\E[w(\cT_1',\cT_1'')]
&=  
\frac{\gamma'}{d}(1+1/\gamma)
+1-\frac{2+\gamma'}{d}
\nonumber\\
&=1-
\frac{2\gamma-\gamma'}{d\gamma}
\nonumber\\
&\le (1-1/d)w(\cT_0',\cT_0''),\label{E_E2}
\end{align}
since $w(\cT_0',\cT_0'')=1$. 

By \eqref{E_E1a}--\eqref{E_E2}, we may apply 
Theorem \ref{T_PC} with $\alpha=O(1/\gamma d)$. 
Note that, by Theorem \ref{T_MainInt}, it follows that 
$D_w=O((1+1/\gamma)D)=O(n^2)$. 
We conclude that 
the mixing time is bounded by $O(\gamma d\log n)$,
as claimed. 
\end{proof}

%%%%%%%%%%%%%%%%%%%%%%%%%%%%%%%%%%%%%%
%%%%%%%%%%%%%%%%%%%%%%%%%%%%%%%%%%%%%%
%%%%%%%%%%%%%%%%%%%%%%%%%%%%%%%%%%%%%%
%%%%%%%%%%%%%%%%%%%%%%%%%%%%%%%%%%%%%%
%%%%%%%%%%%%%%%%%%%%%%%%%%%%%%%%%%%%%%
%%%%%%%%%%%%%%%%%%%%%%%%%%%%%%%%%%%%%%
\section{Future directions}
\label{S_OPs}

We conclude with a list possibilities 
for future study. 

\begin{enumerate}
\item It remains open to find lower bounds for the mixing time, and 
to determine whether our bounds 
are sharp. In types $A_{n-1}$, $B_n$ and $D_n$, 
we might conjecture so, at least up to logarithmic factors.

Recall that, in these types, 
we have shown that $t_{\rm mix}= O(d\log n)$, for {\it any} score sequence, 
where $d$
is the degree of the interchange graph. 
In type $A_{n-1}$, Sarkar \cite{Sar20} 
has shown that $t_{\rm mix}= \Omega(n^3)$ for a special class
of score sequences with $d=\Theta(n^3)$ and a ``bottleneck'' that is simple to analyze.  
A similar argument also works in the other types $B_n$, $C_n$ and $D_n$. 
Perhaps at least $t_{\rm mix}= \Omega(d)$ can be shown to hold in general. 

As discussed, in type $A_n$, 
Chen, Chang and Wang  \cite{CCW09}
have shown that the interchange graph is the 
hypercube, for some very specific score sequences. 
This shows, at least in some cases, that the bound $t_{\rm mix}= O(d\log n)$
is sharp, {\it with} the logarithmic factor. 

In type $C_n$, on the other hand, we have used a re-weighting of the metric
to show that $t_{\rm mix}= O(\gamma d \log n)$, 
where $\gamma$ is the maximal crystal degree. 
Perhaps other techniques can lead to an improvement. 
However, we think that crystals in type $C_n$ are a 
genuine obstacle, 
so it might be surprising if, in fact, $t_{\rm mix}= O(d\log n)$ also
in this type. 

\item Recall that Theorem \ref{T_MainInt} shows that the interchange 
graphs are connected with diameter $D=O(n^2)$.
It might be of theoretical interest to find a precise
formula for $D$, or at least good bounds, as a function of ${\bf s}$. 
Note that Theorem \ref{T_deg} above (proved in \cite{KMP23})
gives such a formula for the degree $d$. We also note that 
in \cite{BL84} some results are proved about the 
diameter $D$ in type $A_{n-1}$.

\item Recall that each edge in the interchange graph corresponds
to a generator reversal. Generators are the smallest neutral structures. 
It might be interesting to consider a generalization, in which 
neutral structures up to a given size can be reversed in a single step, 
and to quantify the decrease  
in the mixing time. 

Related to this, Gioan \cite{Gio07}
has studied cycle and cocycle reversing systems, 
and these have been generalized by  
Backman \cite{Bac17,Bac18}. 
One might pursue Coxeter 
analogues of these results.

\item A {\it graphical zonotope}
is a polytope obtained as a Minkowski sum of line segments, where
the sum is indexed by the edges of the graph (see, e.g., 
Ziegler \cite{Zie95}). 
The permutahedron 
\[
\Pi_{n-1}={\bf w}_n+\sum_{1\le i<j\le n}[{\bf 0}_n,{\bf e}_j-{\bf e}_i]
\] 
is the graphical 
zonotope of the complete graph $K_n$. Likewise, 
the Coxeter versions $\Pi_\Phi$ 
are obtained as sums indexed by the edges in the complete signed
graphs $\cK_\Phi$. 
It could be interesting to study random walks on the fibers
of other graphical zonotopes. 

That being said, our current
arguments take full  advantage of the symmetry of 
$\cK_\Phi$. Once some edges become unavailable, it is more challenging
(or even impossible) 
to show connectivity (and bound the diameter) of the interchange graph, 
and to devise a path coupling (which we have
accomplished, via a non-trivial edge pairing argument).

\item Rapid mixing 
can be a starting point for approximate counting. 
It would be interesting if our result could help with 
counting the number of vertices in interchange graphs, 
for a general score sequence. As already discussed, these 
have been approximated (see \cite{Spe74,McKay90,McKW96,IIMcK20}) 
only in type $A_{n-1}$ and 
when $\bf s$ is close to the center of $\Pi_{n-1}$. 

\item In this work, we have studied random walks
on interchange graphs associated with score sequences.  
However, one could also, quite naturally, try to study random walks 
on the set of score sequences itself.
In type $A_{n-1}$, all lattice points are score sequences. 
In types $B_n$, $C_n$ and $D_n$ the score sequences
are more complicated sets of points, characterized in 
the previous work in this series \cite{KMP23}. 

\item Finally, we recall that the interchange graphs in 
Figures \ref{F_drum} and \ref{F_tamb} are Cartesian products. 
Also recall that, in type $A_n$, some interchange 
graphs are the hypercube \cite{CCW09}, which are a simple 
example of a product graph. 
It might be enlightening to investigate the 
product structure of interchange graphs
more generally. 

\end{enumerate}

%%%%%%%%%%%%%%%%%%%%%%%%%%%%%%%%%%%%%%
%%%%%%%%%%%%%%%%%%%%%%%%%%%%%%%%%%%%%%
%%%%%%%%%%%%%%%%%%%%%%%%%%%%%%%%%%%%%%
%%%%%%%%%%%%%%%%%%%%%%%%%%%%%%%%%%%%%%
%%%%%%%%%%%%%%%%%%%%%%%%%%%%%%%%%%%%%%
%%%%%%%%%%%%%%%%%%%%%%%%%%%%%%%%%%%%%%
\section{Acknowledgements}
We thank 
Christina Goldschmidt, 
James Martin, 
Sam Olesker-Taylor, 
Oliver Riordan and 
Matthias Winkel for helpful conversations. 
We thank the two anonymous reviewers 
for their many insightful comments
and suggestions for future work. 

This publication is based on work (MB; RM) partially supported 
by the EPSRC Centre for Doctoral Training in Mathematics of 
Random Systems: Analysis, Modelling and Simulation (EP/S023925/1). 

TP was supported by the Additional Funding Programme 
for Mathematical Sciences, delivered by EPSRC (EP/V521917/1) 
and the Heilbronn Institute for Mathematical Research.

This work was carried out while BK was at the University of Oxford.
BK gratefully acknowledges the support provided 
by a Florence Nightingale Bicentennial Fellowship
(Department of Statistics) and a Senior Demyship (Magdalen College).

%%%%%%%%%%%%%%%%%%%%%%%%%%%%
%%%%%%%%%%%%%%%%%%%%%%%%%%%%
%%%%%%%%%%%%%%%%%%%%%%%%%%%%
%%%%%%%%%%%%%%%%%%%%%%%%%%%%
%%%%%%%%%%%%%%%%%%%%%%%%%%%%
%%%%%%%%%%%%%%%%%%%%%%%%%%%%
%%%%%%%%%%%%%%%%%%%%%%%%%%%%

\makeatletter
\renewcommand\@biblabel[1]{#1.}
\makeatother

\providecommand{\bysame}{\leavevmode\hbox to3em{\hrulefill}\thinspace}
\providecommand{\MR}{\relax\ifhmode\unskip\space\fi MR }
% \MRhref is called by the amsart/book/proc definition of \MR.
\providecommand{\MRhref}[2]{%
  \href{http://www.ams.org/mathscinet-getitem?mr=#1}{#2}
}
\providecommand{\href}[2]{#2}

\end{document}